\documentclass[11pt]{article}
\usepackage{amsthm, amsmath, amssymb, amsfonts, url, booktabs, tikz, setspace, fancyhdr, bm}
\usepackage{geometry}
\usepackage{hyperref, enumerate}
\usepackage[shortlabels]{enumitem}
\usepackage[babel]{microtype}
\usepackage[english]{babel}
\usepackage[capitalise]{cleveref}
\usepackage{comment}
\usepackage{bbm}
\usepackage{csquotes}
\usepackage{mathabx}
\usepackage{tikz}
\usepackage{graphicx}
\usepackage{float}
\usepackage{xcolor}
\usepackage{mathtools}
\usepackage{mathrsfs}

\usetikzlibrary{positioning, arrows.meta, shapes.geometric}

\counterwithin{figure}{section}

\newtheorem{theorem}{Theorem}[section]

\newtheorem{lemma}[theorem]{Lemma}
\newtheorem{cor}[theorem]{Corollary}

\newtheorem{claim}[theorem]{Claim}

\usetikzlibrary{decorations.pathmorphing}

\theoremstyle{definition}

\newtheorem*{defn-non}{Definition}

\newlist{Case}{enumerate}{2}
\setlist[Case, 1]{%
    label           =   {\bfseries Case \arabic*.},
    labelindent=1em ,labelwidth=1.3cm, labelsep*=1em, leftmargin =!
}
\setlist[Case, 2]{%
    label           =   {\bfseries Subcase \arabic{Casei}.\arabic*.},
    labelindent=-1em ,labelwidth=1.3cm, labelsep*=1em, leftmargin =!
}

\newenvironment{poc}{\begin{proof}[Proof of the claim]}{\end{proof}}

\newcommand{\qbinom}[2]{\genfrac{(}{)}{0pt}{}{#1}{#2}_q}
\newcommand{\diam}{\mathrm{diam}}
\newcommand{\supp}{\mathrm{supp}}

\usepackage{todonotes}

\title{Kleitman's theorem over vector spaces: parity phenomena in canonical and global stability
} 
\author{
Chenhui Lv\thanks{
School of Mathematical Sciences, University of Science and Technology of China, Hefei, 230026, People's Republic of China. Email:lch1994@mail.ustc.edu.cn.}
\and
Zixiang Xu\thanks{Email:zxxu8023@qq.com.}
}

\date{}
\begin{document}
\maketitle
\begin{abstract}
In 1966, Kleitman determined the maximum size of a family of subsets of $[n]$ with bounded symmetric difference. Liao, Liu and Yan recently established a vector-space analogue in the cases $n=d+1$ and $n>2d$, and asked for the sharp bound in the remaining range. We resolve this problem completely by proving the exact vector-space analogue of Kleitman's theorem for every $n\ge d+1$, and we also determine all extremal configurations.

We further develop a stability theory for the vector-space diameter problem. Unlike the Boolean cube, the lattice of subspaces has no translation symmetry, and this makes the stability theory substantially different from its classical counterpart. The geometry of subspace balls leads to two natural notions: canonical stability, which forbids containment only in the canonical extremal configurations, and global stability, which forbids containment in arbitrary balls or adjacent double balls of the corresponding radius. We determine sharp canonical stability in even diameter, sharp canonical and global stability in odd diameter, and prove a nontrivial general upper bound for global stability in even diameter. In particular, these two notions exhibit a sharp parity split: in odd diameter they collapse to the same problem, whereas in even diameter they lead to genuinely different extremal behavior.
\end{abstract}

\section{Introduction} 
Problems of isodiametric type arise in many areas of mathematics~\cite{1992CombinatoricaACZ,1998AMAR,1993DMBL,1990DCGDK,1980SIAMADMFF}, where one seeks the largest possible size of a set with prescribed diameter in a given space. In the discrete setting, a fundamental example asks how large a subset of the Hamming cube can be if all pairwise Hamming distances are bounded. This problem was proposed by Erd\H{o}s and solved by Kleitman~\cite{K1966}.
\begin{theorem}[Kleitman~\cite{K1966}]\label{thm:Kleitman}
Let $n>d\ge 0$ be integers, and let $\mathcal F\subseteq 2^{[n]}$ satisfy
$|F\triangle G|\le d$ for any \(F,G\in\mathcal{F}\).
Then
\[
|\mathcal F|\le
\begin{cases}
\displaystyle \sum_{i=0}^{t}\binom{n}{i}, & \text{if } d=2t,\\[2mm]
\displaystyle \sum_{i=0}^{t}\binom{n}{i}+\binom{n-1}{t}, & \text{if } d=2t+1.
\end{cases}
\]
\end{theorem}

Theorem~\ref{thm:Kleitman} is sharp. When $d=2t$ and $n \geq d+2$, equality is attained by the Hamming ball
\(
\mathcal K(n,d):=\bigcup_{i=0}^{t}\binom{[n]}{i}.
\)
When $d=2t+1$ and $n \geq d+2$, fixing $y\in [n]$, equality is attained by
\(
\mathcal K_y(n,d):=
\bigcup_{i=0}^{t}\binom{[n]}{i}
\cup
\bigl\{A\in \tbinom{[n]}{t+1}: y\in A\bigr\}.
\)
As one of the foundational results in extremal set theory, it is natural to seek analogues in other discrete geometries. Among the most important $q$-analogues of the Boolean lattice is the lattice of subspaces of $\mathbb F_q^n$. More broadly, many central extremal theorems for set systems admit finite vector-space analogues, including the Erd\H{o}s--Ko--Rado theorem~\cite{FW1986,H1975}, the Hilton--Milner theorem~\cite{BBCFMPS2010}, Katona-type theorems~\cite{2013VectorKatona}, and results on cross-intersecting families~\cite{WZ2013}. We also refer the reader to
\cite{2023JCTACao,2022SIDMACao,2023DmShanZhou,2023GCShanZhou,2023JCTAWangXuZhang,2025JCTAYao}
for several further developments in extremal set theory over finite vector spaces.

Let $q$ be a prime power, let $V:=\mathbb F_q^n$, and let $\mathcal V$ denote the lattice of all subspaces of $V$. For $U,W\in\mathcal{V}$, write $U\le W$ if $U$ is a subspace of $W$, and write $U\not\le W$ otherwise. For $0\le k\le n$, write
\(
\mathcal V(k):=\{U\le V:\dim U=k\},\)
we then have
\(
|\mathcal V(k)|=\qbinom{n}{k}
=\prod_{i=0}^{k-1}\frac{q^{\,n-i}-1}{q^{\,k-i}-1}.
\)
In particular, we write \(\boldsymbol{0}\) for the zero subspace of \(V\).

We next introduce the distance that naturally governs our problem, which was used in~\cite{LLY2026}. For $U,W \in \mathcal{V}$, define $\Delta: \mathcal{V}\times \mathcal{V}\to \mathbb{N}$ by
\begin{equation*}
\Delta(U,W)
:=
\dim(U+W)-\dim(U\cap W)
=
\dim U+\dim W-2\dim(U\cap W),
\end{equation*}
which is the natural subspace analogue of the Hamming distance. It is easy to check that $(\mathcal{V}, \Delta)$ is a metric space; see \cite{LLY2026}.
For a family $\mathcal F\subseteq \mathcal V$, write
\[
\diam(\mathcal F):=\max\{\Delta(U,W):U,W\in\mathcal F\},
\quad
\mathcal F(k):=\mathcal F\cap \mathcal V(k),
\quad
\supp(\mathcal{F}) := \{ 0\le i\le n : \mathcal{F}(i) \neq \varnothing \}.
\]
We also write $U^\perp$ for the orthogonal complement of a subspace $U\le V$ with respect to a fixed nondegenerate bilinear form on $V$, and set
\(
\mathcal F^\perp:=\{U^\perp:U\in\mathcal F\}.
\)
For \(C\in\mathcal V\) and \(r\ge0\), define the radius-\(r\) ball
\(
\mathcal B(C,r):=\{A\in\mathcal V:\Delta(A,C)\le r\}.
\)
For \(C_1,C_2\in\mathcal V\), define
\(
\mathcal B(C_1,C_2,r):=\mathcal B(C_1,r)\cup \mathcal B(C_2,r).
\)

With this notation, the natural vector-space analogue of Theorem~\ref{thm:Kleitman} asks for the largest family of subspaces of $V$ with diameter at most $d$ in the metric space $(\mathcal{V}, \Delta)$. A recent result of Liao, Liu and Yan~\cite{LLY2026} established the sharp bound in the cases $n=d+1$ and $n>2d$, and explicitly asked whether the same conclusion remains valid for all remaining parameters. Our first main theorem answers this question completely: for every $n\ge d+1$, the sharp upper bound is exactly the same as in the set case, with ordinary binomial coefficients replaced by Gaussian binomial coefficients, and the extremal families can be fully characterized.

\begin{theorem}\label{thm:main}
Fix a prime power $q$, and let $n\ge d+1\ge 3$.
Let $\mathcal F\subseteq \mathcal V$ satisfy $\diam(\mathcal F)\le d$.
Then
\[
|\mathcal F|\le
\begin{cases}
\displaystyle \sum_{i=0}^{t}\qbinom{n}{i}, & \text{if } d=2t,\\[2mm]
\displaystyle \sum_{i=0}^{t}\qbinom{n}{i}+\qbinom{n-1}{t}, & \text{if } d=2t+1.
\end{cases}
\]
Moreover, equality holds if and only if, after possibly replacing $\mathcal F$ by $\mathcal F^\perp$, one of the following holds:
\begin{enumerate}
\item[\textup{(1)}] $n=d+1$ and $d=2t$.
For $0\le k\le t$, either
\(
|\mathcal F(k)|=\qbinom{n}{k}\)
and
\(|\mathcal F(n-k)|=0,
\)
or
\(
|\mathcal F(k)|=0\) and
\(
|\mathcal F(n-k)|=\qbinom{n}{k}.
\)

\item[\textup{(2)}] $n=d+1$ and $d=2t+1$.
For $0\le k\le t$, either
\(
|\mathcal F(k)|=\qbinom{n}{k}\)
and \(
|\mathcal F(n-k)|=0,
\)
or
\(
|\mathcal F(k)|=0\)
and \(
|\mathcal F(n-k)|=\qbinom{n}{k}.
\)
Moreover,
\(
|\mathcal F(t+1)|=\qbinom{n-1}{t}
\)
and $\mathcal F(t+1)$ is $1$-intersecting in $V$.

\item[\textup{(3)}] $n\ge d+2$, $d=2t$, and 
\(
\mathcal F=\bigcup_{k=0}^{t}\mathcal V(k).
\)

\item[\textup{(4)}] $n\ge d+2$ and $d=2t+1$.
There exists a $1$-dimensional subspace $X\in\mathcal V(1)$ such that
\[
\mathcal F
=
\bigcup_{k=0}^{t}\mathcal V(k)
\;\cup\;
\{A\in \mathcal V(t+1): X\le A\}.
\]
\end{enumerate}
\end{theorem}

Having established the full vector-space analogue of Kleitman's theorem, it is natural to investigate the corresponding stability theory. In the set-theoretic setting, this direction was initiated by Frankl~\cite{F2017}, who not only characterized all extremal families in Kleitman's theorem, but also determined the largest families of bounded diameter that are not contained in an extremal configuration. More recently, Wu, Li, Feng, Liu and Yu~\cite{2025Li} pushed this further by establishing a second-level stability theorem, thereby refining the picture of near-extremal families in the Boolean setting. We also note that analogous higher-order stability phenomena have been studied for other fundamental extremal results, most notably the Erd\H{o}s--Ko--Rado theorem~\cite{2023CombGaoLiuXu,2024AlgeEKR,2017PAMSHanJie,1967Hilton,2024EUJCPeng,2017PAMSMubayi}.

For the stability questions considered in this paper, the relevant range is when the ambient dimension is larger than the diameter parameter by a definite margin. In this regime, the extremal structures given by Theorem~\ref{thm:main} simplify substantially. Indeed, once $n\ge d+2$, the boundary phenomena occurring at $n=d+1$ disappear, and the extremal families are described by a small collection of canonical geometric configurations.

When $d=2t$ is even, the extremal families are precisely the lower and upper radius-$t$ balls
\[
\mathcal L_t:=\mathcal B(\boldsymbol 0,t)=\bigcup_{k=0}^{t}\mathcal V(k),
\qquad
\mathcal U_t:=\mathcal B(V,t)=\mathcal L_t^\perp=\bigcup_{k=n-t}^{n}\mathcal V(k).
\]
When $d=2t+1$ is odd, the extremal families are precisely the {canonical double balls}
\[
\mathcal D_t(X):=\mathcal B(\boldsymbol 0,X,t)
=
\bigcup_{k=0}^{t}\mathcal V(k)\cup \{A\in\mathcal V(t+1):X\le A\},
\qquad X\in\mathcal V(1),
\]
together with their orthogonal images
\(
\mathcal D_t(X)^\perp=\mathcal B(V,X^\perp,t).
\)

These canonical extremal families suggest the most direct analogue of the usual stability question: one asks for the largest families of bounded diameter that are not contained in any of the extremal configurations appearing in Theorem~\ref{thm:main}. This leads to what we call \emph{Type~A} stability, and should be viewed as the primary vector-space analogue of the classical stability problem.

The vector-space setting also gives rise to a stronger stability question. In the Boolean cube, all balls of a fixed radius are equivalent under translations, so excluding the extremal balls is essentially the same as excluding all balls of that radius. This is no longer true for subspaces: the geometry of $\mathcal B(X,t)$ depends on $\dim X$. Thus, besides excluding only the canonical extremal configurations from Theorem~\ref{thm:main}, it is natural to exclude containment in every radius-$t$ ball, or in every union of two adjacent radius-$t$ balls. This is our \emph{Type~B} stability problem.

To state these notions uniformly, let $\mathfrak E$ be a collection of subfamilies of $\mathcal V$. We say that a family $\mathcal F\subseteq\mathcal V$ is \emph{$(\mathfrak{E},d)$-admissible} if
\(
\diam(\mathcal F)\le d\) and
\(
\mathcal F\not\subseteq \mathcal H\) for every
\(\mathcal H\in\mathfrak E.
\)
In other words, $\mathcal F$ has bounded diameter but is not contained in any forbidden configuration from the class $\mathfrak E$. With this language, the stability problems studied in this paper correspond to the following natural choices of $\mathfrak E$.

For even diameter $d=2t$, we define
\[
\mathfrak E_A^{\mathrm{2t}}:=\{\mathcal L_t,\mathcal U_t\},
\qquad
\mathfrak E_B^{\mathrm{2t}}:=\{\mathcal B(X,t):X\in\mathcal V\}.
\]
Thus Type~A excludes only the two canonical extremal balls from Theorem~\ref{thm:main}, whereas Type~B excludes containment in an arbitrary radius-$t$ ball.

For odd diameter $d=2t+1$, we define
\[
\mathfrak E_A^{\mathrm{2t+1}}
:=
\{\mathcal D_t(X),\mathcal D_t(X)^\perp:X\in\mathcal V(1)\},
\qquad
\mathfrak E_B^{\mathrm{2t+1}}
:=
\{\mathcal B(X,Y,t):X,Y\in\mathcal V,\ \Delta(X,Y)=1\}.
\]
Thus Type~A excludes only the canonical double balls appearing in Theorem~\ref{thm:main}, whereas Type~B excludes containment in an arbitrary union of two adjacent radius-$t$ balls.

We first determine the sharp Type~A stability theorem in the even case.

\begin{theorem}\label{thm:stab-typeA-even}
Fix a prime power $q$ and an even integer $d=2t\ge 4$.
Assume that $n\ge 7t+5$.
If $\mathcal F\subseteq \mathcal V$ is $(\mathfrak E_A^{\mathrm{2t}},2t)$-admissible, then
\[
|\mathcal F|
\le
\sum_{i=0}^{t-1}\qbinom{n}{i}
+
\qbinom{n-1}{t-1}
+
\qbinom{n-1}{t}.
\]
Moreover, equality holds if and only if there exists a $1$-dimensional subspace
$X\in\mathcal V(1)$ such that
\(
\mathcal F=\mathcal B(X,t)\)
or
\(
\mathcal F^\perp=\mathcal B(X,t).
\)
\end{theorem}

To describe the odd extremal families, fix $X\in\mathcal V(1)$ and $Y\in\mathcal V(k)$ with $X\not\le Y$, and define
\[
\mathcal{HM}(n,k,X,Y)
:=
\{A\in \mathcal V(k): X\le A,\ \dim(A\cap Y)\ge 1\}
\cup
\{A\in \mathcal V(k): A\le X+Y\}.
\]
When $k=3$, for $Y\in\mathcal V(3)$ define
\[
\mathcal{HM}^*(n,3,Y)
:=
\{A\in \mathcal V(3): \dim(A\cap Y)\ge 2\}.
\]
Set
\[
\mathcal K(n,k,X,Y)
:=
\bigcup_{i=0}^{k-1}\mathcal V(i)\cup \mathcal{HM}(n,k,X,Y),
\]
and, when $k=3$,
\[
\mathcal K^*(n,3,Y)
:=
\bigcup_{i=0}^{2}\mathcal V(i)\cup \mathcal{HM}^*(n,3,Y).
\]

The odd case exhibits an unexpected rigidity phenomenon: although Type~B is a priori stronger than Type~A, the two notions in fact coincide. Accordingly, the sharp odd stability theorem takes the same form for both notions.

\begin{theorem}\label{thm:stab-type-odd}
Fix a prime power $q$ and an odd integer $d=2t+1\ge 5$.
Assume that $n\ge 5t+3$.
Let $\star\in\{A,B\}$. If $\mathcal F\subseteq \mathcal V$ is $(\mathfrak E_\star^{\mathrm{2t+1}},2t+1)$-admissible, then
\[
|\mathcal F|
\le
\sum_{i=0}^{t}\qbinom{n}{i}
+
\qbinom{n-1}{t}
-
q^{t(t+1)}\qbinom{n-t-2}{t}
+
q^{t+1}.
\]
Moreover, equality holds if and only if one of the following occurs:
\begin{enumerate}
\item[\textup{(1)}] There exist $X\in\mathcal V(1)$ and $Y\in\mathcal V(t+1)$ with
$X\not\le Y$ such that
\(
\mathcal F=\mathcal K(n,t+1,X,Y)\) or
\(
\mathcal F^\perp=\mathcal K(n,t+1,X,Y).
\)

\item[\textup{(2)}] $t=2$, and there exists $Y\in\mathcal V(3)$ such that
\(
\mathcal F=\mathcal K^*(n,3,Y)\) or
\(
\mathcal F^\perp=\mathcal K^*(n,3,Y).
\)
\end{enumerate}
\end{theorem}

The remaining case is Type~B in even diameter, which appears to be substantially more delicate. Although we do not determine the exact extremal families in this case, we still obtain a nontrivial general upper bound that is significantly smaller than the Type~A extremal value in Theorem~\ref{thm:stab-typeA-even}.

\begin{theorem}\label{thm:stab-typeB-even}
Fix a prime power $q$ and an even integer $d=2t\ge 4$.
Assume that $n\ge 6t$.
If $\mathcal F\subseteq \mathcal V$ is $(\mathfrak E_B^{\mathrm{2t}},2t)$-admissible, then
\[
|\mathcal F|
<
\sum_{i=0}^{t-1}\qbinom{n}{i}
+
(2t+1)q^2\qbinom{3t}{t}\qbinom{n}{t-1}.
\]
\end{theorem}

For even diameter, to quantify the difference between the Type~A and Type~B bounds, we compare the upper bounds in Theorem~\ref{thm:stab-typeA-even} and Theorem~\ref{thm:stab-typeB-even} for fixed $q$ and $t$.
Beyond the common sum $\sum_{i=0}^{t-1}\qbinom{n}{i}$, which is asymptotically negligible compared to the dominant terms for large $n$, the ratio of their remaining terms satisfies
\begin{align*}
\frac{1}{(2t+1)q^2\qbinom{3t}{t}}
\cdot
\frac{\qbinom{n-1}{t-1}+\qbinom{n-1}{t}}{\qbinom{n}{t-1}}
\ge 
\frac{1}{(2t+1)q^2\qbinom{3t}{t}} \cdot
\frac{(q^{n-t+1}-1)(q^{n-t}-1)}{(q^n-1)(q^t-1)}
\ge c_{q,t} \cdot q^{n},
\end{align*}
where $c_{q,t} > 0$ is a constant depending on $q$ and $t$. For $n \ge 6t$, this ratio grows exponentially with $n$, meaning the Type~A bound asymptotically dominates the Type~B bound. 

Therefore, unlike in the Boolean setting, the vector-space theory leads not to a single stability problem but to a small hierarchy of natural ones. One of the main phenomena uncovered in this paper is that this hierarchy exhibits a sharp parity split: in the even case Type~A and Type~B are genuinely different, whereas in the odd case they coincide.

The remainder of this paper is organized as follows. In Section~\ref{sec:tools}, we collect several technical tools for analyzing the individual layers and complementary layers.
Section~\ref{sec:main-proof} is devoted to the proof of Theorem~\ref{thm:main}, establishing the full vector-space analogue of Kleitman's theorem and characterizing all extremal families.
Section~\ref{sec:stab-even-A} addresses sharp canonical stability in even diameter and proves Theorem~\ref{thm:stab-typeA-even}.
Section~\ref{sec:stab-odd} treats both canonical and global stability in odd diameter and proves Theorem~\ref{thm:stab-type-odd}.
Finally, Section~\ref{sec:stab-even-B} provides a nontrivial upper bound for  global stability in the even diameter case and proves Theorem~\ref{thm:stab-typeB-even}.

\section{Useful tools}\label{sec:tools}

In this section we collect several tools that will be used repeatedly in the proofs of our extremal and stability results. 
We begin with the metric structure on $\mathcal{V}$ and explain how a diameter condition translates into intersection conditions on the individual layers $\mathcal{F}(k)$. 
We then turn to complementary layers, where sharper cross-intersection results yield the main rigidity input. 
Finally, we record the orthogonal-complement isometry, which allows us to normalize the location of the support of $\mathcal{F}$.

\subsection{The layerwise intersection}

We first record the standard counting formula for subspaces with prescribed intersection dimension.

\begin{lemma}[\cite{GM2016}]\label{lem:subspacecounting}
Fix $A \in \mathcal{V}(k)$. Then
\[
\bigl|\{\,B \in \mathcal{V}(\ell): \dim(A \cap B) = j\,\}\bigr|
=
q^{(k-j)(\ell-j)}
\qbinom{n-k}{\ell-j}
\qbinom{k}{j}.
\]
\end{lemma}

To exploit the diameter condition, we use the standard intersection language on Grassmann layers. 
For $\mathcal{A},\mathcal{B}\subseteq\mathcal{V}$, we say that $(\mathcal{A},\mathcal{B})$ is \emph{cross-$s$-intersecting in $V$} if
\(
\dim(A\cap B)\ge s\) for all
\( A\in\mathcal{A},\ B\in\mathcal{B}.
\)
In particular, when $\mathcal{A}=\mathcal{B}$, we say that $\mathcal{A}$ is \emph{$s$-intersecting in $V$}. 
An $s$-intersecting family $\mathcal{A}$ is called \emph{trivial} if all its members contain a common fixed $s$-dimensional subspace of $V$, and \emph{nontrivial} otherwise. 
A trivial $1$-intersecting family is also called a \emph{star}.

The next lemma is the basic mechanism that links the global diameter condition to intersection conditions on the individual layers.

\begin{lemma}[cf. \cite{LLY2026}]\label{lem:intersecting}
Let $\mathcal{F} \subseteq \mathcal{V}$ with $\diam(\mathcal{F}) \le d$. Then for all $i,j\in \supp(\mathcal{F})$, the pair $(\mathcal{F}(i),\mathcal{F}(j))$ is cross-$\left\lceil \frac{i+j-d}{2}\right\rceil$-intersecting in $V$. 
In particular, for every $k\in \supp(\mathcal{F})$, the family $\mathcal{F}(k)$ is $(k-\lfloor d/2\rfloor)$-intersecting in $V$.
\end{lemma}

Lemma~\ref{lem:intersecting} allows us to reduce many layerwise estimates to classical extremal results for intersecting families of subspaces. 
In particular, applying it to a fixed layer $\mathcal{F}(k)$ leads naturally to the vector-space Erd\H{o}s--Ko--Rado theorem.

\begin{theorem}[\cite{FW1986}]\label{thm:ekrvs}
Let $n \ge 2k-s$ and let $\mathcal{A}\subseteq \mathcal{V}(k)$. If $\mathcal{A}$ is $s$-intersecting in $V$, then
\[
|\mathcal{A}|
\le
\max\left\{
\qbinom{n-s}{k-s},
\qbinom{2k-s}{k-s}
\right\}.
\]
\end{theorem}

The maximum size and the extremal families of $1$-intersecting families of $k$-spaces were determined by Hsieh~\cite{H1975} for $n \geq 2k+1$.

\begin{theorem}[\cite{H1975}]\label{thm:EKR-vector-space-1}
Let $n \ge 2k+1$ and let $\mathcal{A} \subseteq \mathcal V(k)$. If $\mathcal{A}$ is $1$-intersecting in $V$, then
\(
|\mathcal{A}| \le \qbinom{n-1}{k-1}.
\)
Moreover, equality holds if and only if there exists a $1$-dimensional subspace $X \in \mathcal V(1)$ such that
$\mathcal{A} = \{A \in \mathcal V(k) : X \le A\}$.
\end{theorem}

The upper bounds for nontrivial $s$-intersecting families in $\mathcal{V}(k)$ have been completely determined in a series of works \cite{BBCFMPS2010,2022SIDMACao,2023JCTAWangXuZhang,2024LAAWangXuYang}. We summarize parts of Theorem~1.4 in \cite{BBCFMPS2010}, Theorem~1.3 in \cite{2022SIDMACao}, Theorem~2.6 in \cite{2023JCTAWangXuZhang},  and Theorem~1.2 in \cite{2024LAAWangXuYang} in the following lemma.

\begin{lemma}\label{lem:nontrivial-s-intersecting}
Let $n, k$ and $s$ be positive integers with $s \geq 1$, $k \geq s + 2$, and $n \geq 2k + 2$. Let $\mathcal{A} \subseteq \mathcal{V}(k)$ be a nontrivial $s$-intersecting family. Then the following hold:
\begin{enumerate}
  \item[(1)] If $1 \leq s \leq \frac{k}{2} - 1$, then
\[
|\mathcal{A}|
\leq
\qbinom{n-s}{k-s}
- q^{(k+1-s)(k-s)} \qbinom{n-k-1}{k-s}
+ q^{k+1-s} \qbinom{s}{1}.
\]

 \item[(2)]  If $\frac{k}{2} - 1 < s \leq k - 2$, then
\[
|\mathcal{A}|
\leq
\qbinom{s+2}{1}\qbinom{n-s-1}{k-s-1}
- q \qbinom{s+1}{1}\qbinom{n-s-2}{k-s-2}.
\]
\end{enumerate}
\end{lemma}

The following lemma is Lemma~2.7 of \cite{2022SIDMACao}, which provides an upper bound for the right-hand side of inequality~(1) in Lemma~\ref{lem:nontrivial-s-intersecting}.

\begin{lemma}[cf. \cite{2022SIDMACao}]\label{lem:nontrivial-s-intersecting-2}
Let $n$, $k$, and $s$ be positive integers with $1 \le s \le k-3$ and $2k \le n$. Then
\[
\qbinom{n-s}{k-s}
- q^{(k+1-s)(k-s)} \qbinom{n-k-1}{k-s}
+ q^{k+1-s} \qbinom{s}{1}
\le
\qbinom{k-s+1}{1}\qbinom{n-s-1}{k-s-1}.
\]
\end{lemma}

The extremal  nontrivial $1$-intersecting families of $k$-spaces were determined in Theorem~1.4 of \cite{BBCFMPS2010}.

\begin{lemma}[cf. \cite{BBCFMPS2010}]\label{lem:nontrivial-1-intersecting}
Let $n$ and $k$ be positive integers with $k \geq 3$ and $n \geq 2k+2$. Let $\mathcal{A} \subseteq \mathcal{V}(k)$ be a nontrivial $1$-intersecting family. Then the following hold:
\[
|\mathcal A|
\le 
\qbinom{n-1}{k-1}
-
q^{k(k-1)}\qbinom{n-k-1}{k-1}
+
q^{k},
\]
with equality if and only if 
\begin{enumerate}
\item[\textup{(1)}]
$\mathcal A=\mathcal{HM}(n,k,X,Y)$ for some
$X\in \mathcal V(1)$ and $Y\in \mathcal V(k)$ with $X\not\le Y$, or
\item[\textup{(2)}]
$\mathcal A=\mathcal{HM}^*(n,3,Y)$ when $k=3$. 
\end{enumerate}
\end{lemma}

\subsection{Complementary layers}

For rough upper bounds, Theorem~\ref{thm:ekrvs} is often sufficient. 
However, when we study equality or near-equality, a more refined analysis is needed. 
The key special case is when two layers are complementary, namely when $i+j=n$. 
In that case, Lemma~\ref{lem:intersecting} yields cross-$1$-intersection, and one can appeal to a much sharper theorem of Wang and Zhang~\cite{WZ2013}.

\begin{theorem}[\cite{WZ2013}]\label{lem:crossintersecting}
Let $n,a,b,t$ be positive integers with
$n\ge 4$, $a,b\ge 2$, $t<\min\{a,b\}$, $a+b<n+t$, and
\(
\qbinom{n}{a}\le \qbinom{n}{b}.
\) Let $V$ be an $n$-dimensional vector space over $\mathbb{F}_q$.
If $\mathcal{A}\subsetneq \mathcal{V}(a)$ and $\mathcal{B}\subsetneq \mathcal{V}(b)$ are cross-$t$-intersecting, then
\begin{equation*}
|\mathcal{A}|+|\mathcal{B}|
\le
\qbinom{n}{b}
-
\sum_{i=0}^{t-1}
q^{(a-i)(b-i)}
\qbinom{a}{i}
\qbinom{n-a}{b-i}
+1.
\end{equation*}
Moreover, equality holds if and only if one of the following holds:
\begin{enumerate}
\item[\textup{(1)}]
$\mathcal{A}=\{A\}$ and
\(
\mathcal{B}
=
\{\,B \in \mathcal{V}(b): \dim(A\cap B)\ge t\,\}
\)
for some $A\in \mathcal{V}(a)$;

\item[\textup{(2)}]
$\qbinom{n}{a}=\qbinom{n}{b}$,
\(
\mathcal{A}
=
\{A \in \mathcal{V}(a): \dim(A\cap B)\ge t\}\)
and \(
\mathcal{B}=\{B\}
\)
for some $B\in \mathcal{V}(b)$.
\end{enumerate}
\end{theorem}

Applying Theorem~\ref{lem:crossintersecting} with $a=k$, $b=n-k$, and $t=1$, we obtain a sharp estimate for the sum of two complementary layers.

\begin{lemma}\label{lem:complementsum-2}
Let $n\ge d+1$ and $1\le k<n-k$. Let $\mathcal{F} \subseteq \mathcal{V}$ with $\diam(\mathcal{F}) \le d$.
If $\mathcal{F}(k)\neq \varnothing$ and $\mathcal{F}(n-k)\neq \varnothing$, then
\[
|\mathcal{F}(k)|+|\mathcal{F}(n-k)|
\le
\qbinom{n}{k}-q^{k(n-k)}+1.
\]
Moreover, equality holds if and only if one of the following holds:
\begin{enumerate}
\item[\textup{(1)}]
$\mathcal{F}(k)=\{X\}$ and
\(
\mathcal{F}(n-k)
=
\{\,Y\in \mathcal{V}(n-k): \dim(X\cap Y)\ge 1\,\}
\)
for some $X\in \mathcal{V}(k)$;

\item[\textup{(2)}]
\(
\mathcal{F}(k)
=
\{\,X\in \mathcal{V}(k): \dim(X\cap Y)\ge 1\,\}
\)
and
\(
\mathcal{F}(n-k)=\{Y\}
\)
for some $Y\in \mathcal{V}(n-k)$.
\end{enumerate}
\end{lemma}

\begin{proof}[Proof of Lemma~\ref{lem:complementsum-2}]
Since $n\ge d+1$, we have
\(
\left\lceil \frac{n-d}{2}\right\rceil \ge 1.
\)
By Lemma~\ref{lem:intersecting}, the pair $(\mathcal{F}(k),\mathcal{F}(n-k))$ is cross-$1$-intersecting in $V$. If \(2\le k<n-k\), then neither layer can be the whole Grassmann layer: if, for instance, \(\mathcal F(k)=\mathcal V(k)\) and \(Y\in\mathcal F(n-k)\), then one can choose \(X\in\mathcal V(k)\) with \(X\cap Y= \boldsymbol 0 \), contradicting cross-$1$-intersection; the other case is similar. Hence both layers are proper subfamilies, and the conclusion follows from Theorem~\ref{lem:crossintersecting}.

It remains to consider the case $k=1<n-k$. 
If $|\mathcal{F}(1)|=1$ or $|\mathcal{F}(n-1)|=1$, then the conclusion follows immediately from Lemma~\ref{lem:subspacecounting}.
So assume that $|\mathcal{F}(1)|\ge 2$ and $|\mathcal{F}(n-1)|\ge 2$. 
Since $(\mathcal{F}(1),\mathcal{F}(n-1))$ is cross-$1$-intersecting, we have $X\le Y$ for every $X\in \mathcal{F}(1)$ and $Y\in \mathcal{F}(n-1)$. Hence
\[
|\mathcal{F}(1)|+|\mathcal{F}(n-1)|
\le
\qbinom{n-2}{1}+\qbinom{n-2}{n-3}
=
2\qbinom{n-2}{1}
<
\qbinom{n-1}{1}
=
\qbinom{n}{1}-q^{n-1}.
\]
\qedhere
\end{proof}

The previous lemma has the following immediate consequence, which shows that the trivial upper bound
\(
|\mathcal{F}(k)|+|\mathcal{F}(n-k)|\le \qbinom{n}{k}
\)
can only be attained in a completely degenerate way.

\begin{cor}\label{lem:complementsum-1}
Let $n\ge d+1$. Let $\mathcal{F} \subseteq \mathcal{V}$ with $\diam(\mathcal{F}) \le d$. For every $0\le k<n-k$, we have
\[
|\mathcal{F}(k)|+|\mathcal{F}(n-k)|
\le
\qbinom{n}{k}.
\]
Moreover, equality holds if and only if either
\(
|\mathcal{F}(k)|=\qbinom{n}{k}\)
and
\(
\mathcal{F}(n-k)=\varnothing,
\)
or
\(
\mathcal{F}(k)=\varnothing\)
and \(
|\mathcal{F}(n-k)|=\qbinom{n}{k}.
\)
\end{cor}
\begin{proof}[Proof of Corollary~\ref{lem:complementsum-1}]
For \(k=0\), the two layers are \(\mathcal V(0)=\{\boldsymbol 0\}\) and \(\mathcal V(n)=\{V\}\). They cannot both meet \(\mathcal F\), since \(\Delta(\boldsymbol 0,V)=n>d\). Hence the stated bound and equality condition are immediate. Now let \(1\le k<n-k\). If both \(\mathcal F(k)\) and \(\mathcal F(n-k)\) are nonempty, Lemma~\ref{lem:complementsum-2} gives
\[
|\mathcal F(k)|+|\mathcal F(n-k)|
\le \qbinom nk-q^{k(n-k)}+1
<\qbinom nk.
\]
Therefore equality in the trivial bound can occur only when one of the two layers is empty, and then the other must be the whole corresponding Grassmann layer.
\end{proof}

Iterating Corollary~\ref{lem:complementsum-1} over all complementary pairs already forces a strong global structure. 
The next lemma records the resulting rigidity in the even-diameter case.

\begin{lemma}\label{lem:equality}
Let $\mathcal{F} \subseteq \mathcal{V}$ with $\diam(\mathcal{F}) \le d$, and put \(t:=\lfloor d/2\rfloor\).
If $n\ge d+1$ and
\(
|\mathcal{F}(k)|+|\mathcal{F}(n-k)|=\qbinom{n}{k}\)
for \(k=0,1,\ldots,t,\)
then either $\mathcal{F}$ or $\mathcal{F}^\perp$ satisfies one of the following:
\begin{enumerate}
\item[\textup{(1)}]
$n=d+1$, and for each $k=0,1,\ldots,t$, either
\(
|\mathcal{F}(k)|=\qbinom{n}{k}\)
and \(
|\mathcal{F}(n-k)|=0,
\)
or vice versa;

\item[\textup{(2)}]
$n\ge d+2$, and
\(
|\mathcal{F}(k)|=\qbinom{n}{k}\)
and
\(
|\mathcal{F}(n-k)|=0
\)
for \(k=0,1,\ldots,t.\)
\end{enumerate}
\end{lemma}

\begin{proof}[Proof of Lemma~\ref{lem:equality}]
Clearly, exactly one of $\mathcal{F}(0)$ and $\mathcal{F}(n)$ is nonempty. 
If $\mathcal{F}(n)\neq \varnothing$, then we pass to $\mathcal{F}^\perp$. 
Thus we may assume that $\mathcal{F}(0)\neq \varnothing$.

By Corollary~\ref{lem:complementsum-1}, for each $0\le k\le t$, either
\(
|\mathcal{F}(k)|=\qbinom{n}{k}\)
and \(
\mathcal{F}(n-k)=\varnothing,
\)
or
\(
\mathcal{F}(k)=\varnothing\) and
\(
|\mathcal{F}(n-k)|=\qbinom{n}{k}.
\)

For part~(1), it suffices to observe that for every $X\in \mathcal{V}(k)$ and every $Y\in \mathcal{V}\setminus \mathcal{V}(n-k)$, we have $\Delta(X,Y)\le d$. 
Indeed, if $\dim Y\le n-k-1$, then
\[
\Delta(X,Y)\le \dim X+\dim Y\le k+(n-k-1)=d.
\]
If $\dim Y\ge n-k+1$, then
\[
\dim(X\cap Y)
=
\dim X+\dim Y-\dim(X+Y)
\ge
k+(n-k+1)-n
\ge 1,
\]
and hence
\[
\Delta(X,Y)
\le
\dim(X+Y)-\dim(X\cap Y)
\le n-1=d.
\]

For part~(2), we argue by induction on $t$. 
The case $t=0$ is immediate. 
Assume the statement holds for $t-1$. Then
\(
|\mathcal{F}(t-1)|=\qbinom{n}{t-1}\)
and
\(
|\mathcal{F}(n-t+1)|=0.
\)
If $\mathcal{F}(n-t)\neq \varnothing$, choose $Y\in \mathcal{F}(n-t)$. Since
\(
|\mathcal{F}(t-1)|=\qbinom{n}{t-1}=|\mathcal{V}(t-1)|,
\)
we may choose $X\in \mathcal{F}(t-1)$ with $X\cap Y= \boldsymbol 0$. Then
\[
\Delta(X,Y)=\dim X+\dim Y=n-1\ge d+1,
\]
since $n\ge d+2$, a contradiction. Therefore $\mathcal{F}(n-t)=\varnothing$, and so
\(
|\mathcal{F}(t)|=\qbinom{n}{t}.
\)
\qedhere
\end{proof}

In the odd-diameter case there is one additional middle layer, and the previous rigidity statement must be supplemented by a structural description of that layer.

\begin{lemma}\label{lem:equality-odd}
Let \(\mathcal F\subseteq\mathcal V\) satisfy \(\diam(\mathcal F)\le d\).
If $n\ge d+1$, $d=2t+1$, and
\(
|\mathcal{F}(k)|+|\mathcal{F}(n-k)|=\qbinom{n}{k}\)
for \(k=0,1,\ldots,t,\) and furthermore one of the following holds:
\begin{enumerate}
\item[(a)]
\(
|\mathcal{F}(t+1)|=\qbinom{n-1}{t},
\)
$\mathcal{F}(t+1)$ is $1$-intersecting in $V$, and
$\mathcal{F}(n-t-1)=\varnothing$ if $n\neq d+1$;

\item[(b)]
\(
|\mathcal{F}(n-t-1)|=\qbinom{n-1}{t},
\)
$\mathcal{F}(n-t-1)$ is $(n-2t-1)$-intersecting in $V$, and
$\mathcal{F}(t+1)=\varnothing$ if $n\neq d+1$,
\end{enumerate}
then either $\mathcal{F}$ or $\mathcal{F}^\perp$ satisfies one of the following:
\begin{enumerate}
\item[\textup{(1)}]
$n=d+1$, and for each $k=0,1,\ldots,t$, either
\(
|\mathcal{F}(k)|=\qbinom{n}{k}\)
and
\(
|\mathcal{F}(n-k)|=0,
\)
or vice versa. Furthermore,
\(
|\mathcal{F}(t+1)|=\qbinom{n-1}{t},
\)
and
\(
\mathcal{F}(t+1)=\mathcal{F}(n-t-1)
\)
is $1$-intersecting in $V$;

\item[\textup{(2)}]
$n\ge d+2$, and
\(
|\mathcal{F}(k)|=\qbinom{n}{k},\)
\(
|\mathcal{F}(n-k)|=0\)
for \(k=0,1,\ldots,t.\)
Furthermore,
\(
|\mathcal{F}(t+1)|=\qbinom{n-1}{t},
\)
and
\(
\mathcal{F}(t+1)=\{\,A\in \mathcal{V}(t+1): X\le A\,\}
\)
for some fixed $1$-dimensional subspace $X \in \mathcal{V}(1)$.
\end{enumerate}
\end{lemma}

\begin{proof}[Proof of Lemma~\ref{lem:equality-odd}]
The two assumptions~(a) and~(b) are dual to each other under the orthogonal-complement map. 
Thus, after possibly replacing $\mathcal{F}$ by $\mathcal{F}^\perp$, we may assume that $\mathcal{F}(0)\neq \varnothing$. 

If $n=d+1=2t+2$, then $\mathcal{F}(t+1)=\mathcal{F}(n-t-1)$ is $1$-intersecting in $V$. 
The desired conclusion then follows from Lemma~\ref{lem:equality}.

Assume now that $n\ge d+2$. Then Lemma~\ref{lem:equality} gives
\(
|\mathcal{F}(k)|=\qbinom{n}{k}\)
and
\(
|\mathcal{F}(n-k)|=0
\)
for \(k=0,1,\ldots,t.\)
Assume that $\mathcal{F}(n-t-1) \neq \varnothing$, and let $Y \in \mathcal{F}(n-t-1)$. Since
\(
|\mathcal{F}(t)| = \qbinom{n}{t} = |\mathcal V(t)|,
\)
we may choose $X \in \mathcal{F}(t)$ such that $X \cap Y = \boldsymbol{0}$. As $n \ge d+2$, it follows that
\[
\Delta(X,Y) = \dim X + \dim Y = n-1 \ge d+1,
\]
a contradiction. Thus, $\mathcal{F}(n-t-1) = \varnothing$, and $|\mathcal{F}(t+1)| = \qbinom{n-1}{t}$; moreover, $\mathcal{F}(t+1)$ is $1$-intersecting in $V$.
Since
\(
n\ge d+2=2t+3,
\) 
Theorem \ref{thm:EKR-vector-space-1} implies that
\[
\mathcal{F}(t+1)=\{\,A\in \mathcal{V}(t+1): X \le A\,\}
\]
for some fixed $1$-dimensional subspace $X \in \mathcal{V}(1)$.
\qedhere
\end{proof}

\subsection{Orthogonal complements and normalization of the support}

The final tool in this section is the orthogonal-complement map. 
Its role is not to provide a direct counting bound, but rather to allow us to move freely between lower layers and upper layers while preserving all distance information. 
This gives a convenient normalization of the support of $\mathcal{F}$.

Recall that if $(X,\delta)$ and $(Y,\rho)$ are metric spaces, then a bijection $f:X\to Y$ is called an \emph{isometry} if
\(
\rho(f(x_1),f(x_2))=\delta(x_1,x_2)\)
for all \(x_1,x_2\in X.\)

\begin{lemma}[\cite{LLY2026}]\label{lem:orthogonal}
The bijection $f(U):=U^\perp$ is an isometry on $(\mathcal{V},\Delta)$.
\end{lemma}

Given $\mathcal{F}\subseteq \mathcal{V}$, recall that
\(
\diam(\mathcal{F})
:=
\max_{A,B\in \mathcal{F}}\Delta(A,B),
\)
and define
\[
D(\mathcal{F})
:=
\max_{A,B\in \mathcal{F}} |\dim A-\dim B|.
\]
Since
\(
|\dim A-\dim B|\le \Delta(A,B)
\)
for all \(A,B\in \mathcal{F},\)
we always have
\(
D(\mathcal{F})\le \diam(\mathcal{F}).
\)
Define
\[
m_{\mathcal{F}}
:=
\min\{
\supp(\mathcal{F})
\cup 
\supp(\mathcal{F}^\perp)\}.
\]

The next elementary lemma shows that, after possibly passing to $\mathcal{F}^\perp$, the support interval of $\mathcal{F}$ may always be assumed to start from a value at most $\frac{n}{2}$.

\begin{lemma}\label{lem:min-bound}
Denote $D(\mathcal{F})=D$, then
\(
m_{\mathcal{F}}\le \frac{n-D}{2}.
\)
\end{lemma}

\begin{proof}[Proof of Lemma~\ref{lem:min-bound}]
If $\supp(\mathcal{F}) \subseteq [x,x+D]$, then
\(
\supp(\mathcal{F}^\perp) \subseteq [n-D-x,n-x].
\)
Hence
\[
m_{\mathcal{F}}
\le
\frac{x+(n-D-x)}{2}
=
\frac{n-D}{2}.
\]
\qedhere
\end{proof}

Combining Lemmas~\ref{lem:orthogonal} and \ref{lem:min-bound}, we may therefore assume, whenever convenient, that
\begin{equation*}
m_{\mathcal{F}}
=
\min\supp(\mathcal{F})
\le
\frac{n-D}{2}.
\end{equation*}
This normalization will be used repeatedly in the sequel.

\section{Full version of Kleitman's theorem in vector spaces: Proof of Theorem~\ref{thm:main}}\label{sec:main-proof}
The cases \(d=2,3\) require only minor modifications of the layerwise argument below. They can also be verified by the small-diameter analysis in~\cite{LLY2026}, with the same equality characterization. Hence, in the remainder of the proof we assume that \(d\ge4\) and put
\(
t:=\lfloor d/2\rfloor\ge 2.
\)
 Let
\[
D:=D(\mathcal F)\le \diam(\mathcal F)\le d.
\]
Set $m:=\min \supp(\mathcal{F})$.
By Lemmas~\ref{lem:orthogonal} and~\ref{lem:min-bound}, after possibly replacing $\mathcal F$ by $\mathcal F^\perp$, we may assume that 
\(
\supp(\mathcal F) \subseteq [m,m+D]\) and
\(
m\le \frac{n-D}{2}.
\)
Set
\(
M:=\max\supp(\mathcal F).
\)
Then $M= m+D$, and hence
\begin{equation}\label{eq:mM}
m+M= 2m+D\le n.
\end{equation}

We shall estimate $\mathcal F$ layer by layer. 
By Lemma~\ref{lem:intersecting}, for every $k\in\supp(\mathcal F)$ the layer $\mathcal F(k)$ is $(k-t)$-intersecting in $V$. 
Thus, whenever $k\ge t$ and $n\ge k+t$, Theorem~\ref{thm:ekrvs} gives
\begin{equation}\label{eq:slice-bound}
|\mathcal F(k)|
\le
\max\left\{
\qbinom{n-k+t}{t},
\qbinom{k+t}{t}
\right\}.
\end{equation}

We now divide the proof according to the position of the support of $\mathcal F$.

\paragraph{Case 1: $m\ge t+1$.}

In this case, \eqref{eq:mM} gives
\(
n\ge m+M\ge M+t+1.
\)
In particular, $n\ge k+t+1$ for every $k\le M$, so \eqref{eq:slice-bound} applies to every layer of $\mathcal F$.

We first estimate the part of the support at or below $n/2$. 
For $k\le \lfloor n/2\rfloor$, \eqref{eq:slice-bound} gives
\(
|\mathcal F(k)|\le \qbinom{n-k+t}{t}.
\)
Therefore
\begin{align}
\qbinom{n}{t}^{-1}
\sum_{m\le k\le \lfloor n/2\rfloor}|\mathcal F(k)|
&\le
\sum_{k=t+1}^{\lfloor n/2\rfloor}
\frac{\qbinom{n-k+t}{t}}{\qbinom{n}{t}} \notag\\
&=
\sum_{k=t+1}^{\lfloor n/2\rfloor}
\frac{(q^{n-k+t}-1)\cdots(q^{n-k+1}-1)}
{(q^n-1)\cdots(q^{n-t+1}-1)} \notag\\
&<
\sum_{k=t+1}^{\lfloor n/2\rfloor}
q^{-t(k-t)}
<
\sum_{r=1}^{\infty}q^{-tr}
=
\frac{1}{q^t-1}.
\label{eq:case1-low}
\end{align}

Similarly, for $k\ge \lfloor n/2\rfloor+1$, \eqref{eq:slice-bound} gives
\(
|\mathcal F(k)|\le \qbinom{k+t}{t}.
\)
Using $k\le M\le n-t-1$, we get
\begin{align}
\qbinom{n}{t}^{-1}
\sum_{\lfloor n/2\rfloor< k\le M}|\mathcal F(k)|
&\le
\sum_{\lfloor n/2\rfloor<k\le M}
\frac{\qbinom{k+t}{t}}{\qbinom{n}{t}} \notag\\
&=
\sum_{\lfloor n/2\rfloor<k\le M}
\frac{(q^{k+t}-1)\cdots(q^{k+1}-1)}
{(q^n-1)\cdots(q^{n-t+1}-1)} \notag\\
&\le
\sum_{\lfloor n/2\rfloor<k\le M}
q^{-t(n-k-t)}
<
\sum_{r=1}^{\infty}q^{-tr}
=
\frac{1}{q^t-1}.
\label{eq:case1-high}
\end{align}
Combining \eqref{eq:case1-low} and \eqref{eq:case1-high}, we obtain
\begin{equation*}
|\mathcal F|
<
\frac{2}{q^t-1}\qbinom{n}{t}
\le
\frac{2}{3}\qbinom{n}{t}.
\end{equation*}
This is strictly smaller than
\(
\sum_{i=0}^{t}\qbinom{n}{i},
\)
and hence also strictly smaller than the odd-diameter extremal value. 
Thus, the equality case cannot occur in Case~1.

\medskip
Henceforth we may assume \(m\le t.\)
Let
\(
r:=\left\lceil \frac d2\right\rceil.
\)
Thus $r=t$ if $d=2t$, and $r=t+1$ if $d=2t+1$.

\paragraph{Case 2: There is a middle layer.}

Assume that
\begin{equation}\label{eq:middle-layer-exists}
\supp(\mathcal F)\cap [r+1,n-t-1]\neq \varnothing.
\end{equation}
Choose
\(
\xi\in \supp(\mathcal F)\cap [r+1,n-t-1].
\)
We shall show that this forces a strict loss from the extremal bound, except for a boundary odd case which is still handled by Lemma~\ref{lem:equality-odd}.

First decompose
\begin{align*}
|\mathcal F|
&=
\sum_{k=0}^{t-1}
\bigl(|\mathcal F(k)|+|\mathcal F(n-k)|\bigr)
+
\bigl(|\mathcal F(t)|+|\mathcal F(n-t)|\bigr)
+
\sum_{k=t+1}^{n-t-1}|\mathcal F(k)|.
\end{align*}
By Corollary~\ref{lem:complementsum-1},
\begin{equation}\label{eq:lower-pairs}
\sum_{k=0}^{t-1}
\bigl(|\mathcal F(k)|+|\mathcal F(n-k)|\bigr)
\le
\sum_{k=0}^{t-1}\qbinom{n}{k}.
\end{equation}

We now estimate the remaining terms.

\paragraph{First estimate: the pair of layers $\mathcal F(t)$ and $\mathcal F(n-t)$.}

Suppose first that $\mathcal F(n-t)=\varnothing$. 
Fix $Y\in\mathcal F(\xi)$. 
For every $X\in\mathcal F(t)$, the diameter condition gives
\(
\Delta(X,Y)\le d,
\)
and hence
\[
2\dim(X\cap Y)
\ge
t+\xi-d.
\]
Since $\xi\ge r+1$, every $X\in\mathcal F(t)$ meets $Y$ nontrivially. 
By Lemma~\ref{lem:subspacecounting},
\begin{equation*}
|\mathcal F(t)|
\le
\qbinom{n}{t}
-
q^{\xi t}\qbinom{n-\xi}{t}.
\end{equation*}
Moreover,
\begin{align*}
\frac{q^{\xi t}\qbinom{n-\xi}{t}}{\qbinom{n}{t}}
&=
\frac{(q^n-q^\xi)(q^{n-1}-q^\xi)\cdots(q^{n-t+1}-q^\xi)}
{(q^n-1)(q^{n-1}-1)\cdots(q^{n-t+1}-1)}
\notag\\
&\ge
\prod_{i=1}^{t}\left(1-q^{-(n-\xi-t+i)}\right)
\ge
\prod_{i=1}^{t}\left(1-q^{-(i+1)}\right)
\notag\\
&\ge
\prod_{i=1}^{t}\left(1-2^{-(i+1)}\right).
\end{align*}
For $t=2$, the final product is
\(
\left(1-\frac14\right)\left(1-\frac18\right)=\frac{21}{32}.
\)
For $t\ge 3$, we use the standard estimate
\(
\prod_{i=1}^{t}\left(1-2^{-(i+1)}\right)\ge \frac12.
\)
Thus
\begin{equation}\label{eq:c1-A}
|\mathcal F(t)|+|\mathcal F(n-t)|
\le
c_1\qbinom{n}{t},
\qquad
c_1:=
\begin{cases}
11/32, & t=2,\\
1/2, & t\ge 3.
\end{cases}
\end{equation}

The dual argument gives the corresponding estimate when $\mathcal F(t)=\varnothing$, provided that either $d=2t$, or $d=2t+1$ and $\xi\le n-t-2$. 
Indeed, in this case $\dim Y^\perp=n-\xi$ is at least $t+1$ in the even case and at least $t+2$ in the odd case. 
For every $X\in\mathcal F(n-t)$, put $Z=X^\perp\in\mathcal V(t)$. 
Since $\Delta(Z,Y^\perp)=\Delta(X,Y)\le d$, we get
\(
\dim(Z\cap Y^\perp)\ge 1.
\)
Hence
\(
|\mathcal F(n-t)|
\le
\qbinom{n}{t}
-
q^{(n-\xi)t}\qbinom{\xi}{t}.
\)
The same estimate as above gives
\begin{equation}\label{eq:c1-B}
|\mathcal F(t)|+|\mathcal F(n-t)|
\le
c_1\qbinom{n}{t}.
\end{equation}

Finally, suppose that both $\mathcal F(t)$ and $\mathcal F(n-t)$ are nonempty. 
By Lemma~\ref{lem:complementsum-2},
\[
|\mathcal F(t)|+|\mathcal F(n-t)|
\le
\qbinom{n}{t}-q^{t(n-t)}+1.
\]
Furthermore,
\begin{equation*}
\frac{q^{t(n-t)}}{\qbinom{n}{t}}
=
\frac{(q^n-q^{n-t})(q^{n-1}-q^{n-t})\cdots(q^{n-t+1}-q^{n-t})}
{(q^n-1)(q^{n-1}-1)\cdots(q^{n-t+1}-1)}
\ge
\prod_{i=1}^{t}(1-q^{-i}).
\end{equation*}
Set $f(x) := \prod_{i=1}^{\infty} (1 - x^i)^{-1}$ for $x < \tfrac{2}{3}$. By the discussion in \cite{LLY2026}, we have $f(x) \leqslant \frac{2}{2 - 3x}$. Therefore, if $q \geqslant 3$, then $\prod_{i=1}^{\infty} (1 - q^{-i}) \geqslant \frac{1}{f(1/3)} \geqslant \frac{1}{2}$. Also notice that as $\prod_{i=1}^{\infty} (1 - 2^{-i}) \approx 0.288788095$ is a constant (see~\cite{B1980}), for $q=2$ we have
$\prod_{i=1}^{\infty} \left(1 - 2^{-i}\right) \geqslant 0.2887$.

Therefore, there is some constant $c_2=c_2(q,t,n,d)$ such that
\begin{equation}\label{eq:c2}
|\mathcal F(t)|+|\mathcal F(n-t)|
\le
c_2\qbinom{n}{t}+1,
\end{equation}
where we may take
\[
c_2=
\begin{cases}
1/8, & q=2,\ t=2,\ n\ge d+3,\\
0.7113, & q=2,\ t=2,\ n=d+2,\\
43/64, & q=2,\ t=3,\\
0.7113, & q=2,\ t\ge 4,\\
1/2, & q\ge 3.
\end{cases}
\]
Here the special improvement $c_2=1/8$ for $q=2$, $t=2$, and $n\ge d+3$ follows because in that case the pair $(\mathcal F(2),\mathcal F(n-2))$ is cross-$2$-intersecting, and hence every $2$-space in $\mathcal F(2)$ is contained in every $(n-2)$-space in $\mathcal F(n-2)$. Thus
\[
|\mathcal F(2)|+|\mathcal F(n-2)|
\le
2\qbinom{n-2}{2}
<
\frac18\qbinom{n}{2}.
\]

\paragraph{Second estimate: the interior layers.}

Let
\[
\mathcal G:=
\{A\in\mathcal F:t+1\le \dim A\le n-t-1\}.
\]
We claim that
\begin{equation}\label{eq:c3}
|\mathcal G|
\le
c_3\qbinom{n}{t},\qquad 
c_3=
\begin{cases}
41/64, & q=t=2,\\
2/7, & q=2,\ t=3,\\
1/4, & q\ge 3,\ \text{or }q=2\text{ and }t\ge 4.
\end{cases}
\end{equation}

To prove this, let $d_{\mathcal G}:=\diam(\mathcal G)$ and $t_{\mathcal G}:=\lfloor d_{\mathcal G}/2\rfloor$. 
If $t_{\mathcal G}=0$, then $|\mathcal G|\le 2$, and \eqref{eq:c3} is immediate. 
Assume $t_{\mathcal G}\ge 1$. 
The support of $\mathcal G$ is contained in $[t+1,n-t-1]$. 
Passing to $\mathcal G^\perp$ if necessary, we may assume that the support interval satisfies the required normalization, namely, $\min\supp(\mathcal G)+\max\supp(\mathcal G)\le n$, in Case~1. 
Since the first nonempty layer of both $\mathcal G$ and $\mathcal G^\perp$ is at least $t+1\ge t_{\mathcal G}+1$, the argument of Case~1 gives
\begin{equation}\label{eq:G-case1}
|\mathcal G|
<
\frac{2}{q^{t_{\mathcal G}}-1}\qbinom{n}{t_{\mathcal G}}.
\end{equation}
If \(t_{\mathcal G}<t\), then \(n\ge d+1\ge 2t+1\) and the monotonicity of Gaussian coefficients give
\(
\frac{2}{q^{t_{\mathcal G}}-1}\qbinom{n}{t_{\mathcal G}}
\le
\frac{2}{q^{t}-1}\qbinom{n}{t}.
\)
Thus \eqref{eq:G-case1} implies
\(
|\mathcal G|
<
\frac{2}{q^t-1}\qbinom{n}{t}.
\)
This is at most $\frac{2}{7}\qbinom{n}{t}$ when $q=2,t=3$, and at most $\frac14\qbinom{n}{t}$ when $q\ge 3$ or $q=2,t\ge 4$.

It remains only to record the slightly sharper estimate needed for $q=t=2$. 
Since $q=t=2$, we have $d\in\{4,5\}$ and hence $D(\mathcal F)\le 5$.
Moreover, in the present case we have $m\le t=2$, while every member of $\mathcal G$
has dimension at least $3$. Therefore the number of possible dimensions occurring in
$\mathcal G$ is at most $D(\mathcal F)\le 5$. 
For a layer $k\le n/2$ in the support of $\mathcal G$, the proof of Case~1 gives
\(
\qbinom{n}{2}^{-1}|\mathcal G(k)|
<
2^{-2(k-2)}.
\)
For a layer $k>n/2$, it gives
\(
\qbinom{n}{2}^{-1}|\mathcal G(k)|
<
2^{-2(n-k-2)}.
\)
Taking the five largest possible contributions gives
\[
|\mathcal G|
\le
\left(
\frac14+\frac14+\frac1{16}+\frac1{16}+\frac1{64}
\right)\qbinom{n}{2}
=
\frac{41}{64}\qbinom{n}{2}.
\]
This proves \eqref{eq:c3}.

\paragraph{Third estimate: putting the previous bounds together.}
Combining \eqref{eq:lower-pairs}, \eqref{eq:c1-A}, \eqref{eq:c1-B}, \eqref{eq:c2}, and \eqref{eq:c3}, we obtain the following conclusion.  Except possibly in the following cases:
\begin{enumerate}
\item[\textup{(1)}]
$q=t=2$, $n=d+2$, and both $\mathcal F(t)$ and $\mathcal F(n-t)$ are nonempty;

\item[\textup{(2)}]
\(d=2t+1\), \(\mathcal F(t)=\varnothing\),
\(\mathcal F(n-t-1)\neq\varnothing\), and
\(\mathcal F(k)=\varnothing\) for all \(t+2\le k\le n-t-2\),
\end{enumerate}
we have
\(
c_2+c_3\le \frac{63}{64}.
\)
Indeed, this is immediate from the above choices of \(c_2\) and \(c_3\): for \(q=t=2\) and \(n\ge d+3\) we have \(\frac18+\frac{41}{64}<\frac{63}{64}\); for \(q=2,t=3\) we have \(\frac{43}{64}+\frac{2}{7}<\frac{63}{64}\); for \(q=2,t\ge4\) we have \(\frac{365}{512}+\frac{1}{4}<\frac{63}{64}\); and for \(q\ge3\) we have \(\frac{1}{2}+\frac{1}{4}<\frac{63}{64}\).
Therefore,
\[
|\mathcal F(t)|+|\mathcal F(n-t)|+
\sum_{k=t+1}^{n-t-1}|\mathcal F(k)|
\le
\frac{63}{64}\qbinom{n}{t}+1.
\]
Since
\(
\qbinom{n}{t}\ge \qbinom{5}{2}\ge 155,
\)
we have
\(
1<\frac1{64}\qbinom{n}{t}.
\)
Therefore
\[
|\mathcal F(t)|+|\mathcal F(n-t)|+
\sum_{k=t+1}^{n-t-1}|\mathcal F(k)|
<
\qbinom{n}{t}.
\]
Together with \eqref{eq:lower-pairs}, this gives
\(
|\mathcal F|
<
\sum_{k=0}^{t}\qbinom{n}{k}.
\)
Hence in this situation $\mathcal F$ is strictly below the even extremal value, and therefore also strictly below the odd extremal value.

It remains to discuss the exceptional case
\(
q=t=2,\) \( n=d+2\)
with both $\mathcal F(t)$ and $\mathcal F(n-t)$ nonempty. 
If $d=2t$, then the only interior layer is $\mathcal F(t+1)$, and
\begin{align}\label{q=t=2-bound}
|\mathcal F(t+1)|\le \qbinom{n-1}{t}<\frac14\qbinom{n}{t}.  
\end{align}
Together with \eqref{eq:c2}, this again gives
\[
|\mathcal F(t)|+|\mathcal F(n-t)|+|\mathcal F(t+1)|
<
\qbinom{n}{t}.
\]
Thus the even case is still strict.

If $d=2t+1$, then the interior layers are $\mathcal F(t+1)$ and $\mathcal F(t+2)$, and each has size at most $\qbinom{n-1}{t}$. 
By \eqref{eq:c2} and \eqref{q=t=2-bound}, we have
\[
|\mathcal F(t)|+|\mathcal F(n-t)|+|\mathcal F(t+1)|+|\mathcal F(t+2)|
<
\qbinom{n}{t}+\qbinom{n-1}{t}.
\]
Hence the odd extremal bound is also strict.

We have proved that if \eqref{eq:middle-layer-exists} holds, then equality in Theorem~\ref{thm:main} cannot occur, except possibly in the following odd boundary configuration:
\(
d=2t+1,\)
\(\mathcal F(t)=\varnothing,\)
\(
\mathcal F(n-t-1)\neq\varnothing\)
and
\(
\mathcal F(k)=\varnothing\)
for all \(t+2\le k\le n-t-2.\)
We now treat this remaining configuration.

If $n=d+1=2t+2$, then $n-t-1=t+1$. 
The family $\mathcal F(t+1)$ is $1$-intersecting in $V$, so by Theorem~\ref{thm:ekrvs},
\(
|\mathcal F(t+1)|\le \qbinom{n-1}{t}.
\)
Using Corollary~\ref{lem:complementsum-1} for $k=0,1,\ldots,t$, we get
\[
|\mathcal F|
\le
\sum_{k=0}^{t}\qbinom{n}{k}
+
\qbinom{n-1}{t}.
\]
Equality holds precisely when
\(
|\mathcal F(k)|+|\mathcal F(n-k)|=\qbinom{n}{k}
\)
for \(k=0,1,\ldots,t,\)
and
\(
|\mathcal F(t+1)|=\qbinom{n-1}{t}.
\)
By Lemma~\ref{lem:equality-odd}, this is exactly the equality case described in Theorem~\ref{thm:main}(2).

Now assume $n\ge d+2$. 
If $\mathcal F(t+1)\neq\varnothing$, then both $\mathcal F(t+1)$ and $\mathcal F(n-t-1)$ have size at most $\qbinom{n-1}{t}$. 
Moreover, the same dual deficit estimate as in \eqref{eq:c1-B} gives
\(
|\mathcal F(n-t)|\le c_1\qbinom{n}{t}.
\)
Consequently,
\[
|\mathcal F|
\le
\sum_{k=0}^{t-1}\qbinom{n}{k}
+
2\qbinom{n-1}{t}
+
c_1\qbinom{n}{t}.
\]
Since
\(
\frac{\qbinom{n-1}{t}}{\qbinom{n}{t}}
=
\frac{q^{n-t}-1}{q^n-1}
<
q^{-t}
\le \frac14,
\)
and $c_1\le \frac12$, the right-hand side is strictly smaller than
\(
\sum_{k=0}^{t}\qbinom{n}{k}
+
\qbinom{n-1}{t}.
\)

Finally suppose $\mathcal F(t+1)=\varnothing$. 
Then
\[
|\mathcal F|
\le
\sum_{k=0}^{t}\qbinom{n}{k}
+
\qbinom{n-1}{t}.
\]
If equality held, then
\(
|\mathcal F(k)|+|\mathcal F(n-k)|=\qbinom{n}{k}\)
for \(k=0,1,\ldots,t,\)
and
\(
|\mathcal F(n-t-1)|=\qbinom{n-1}{t},
\)
with $\mathcal F(n-t-1)$ being $(n-2t-1)$-intersecting in $V$. 
By the assumption $m\le t$, if equality holds, then we claim that $\mathcal{F}(0) \neq \varnothing$. 
Suppose instead that \(\mathcal F(0)=\varnothing\), and let
\(s:=\min\supp(\mathcal F)\). Since \(\mathcal F(t)=\varnothing\) in the present subcase, we have
\(
1 \le s \le t-1.
\) 
Equality in the complementary-pair estimates gives
\(\mathcal F(s)=\mathcal V(s)\). Also, by the minimality of \(s\), we have
\(\mathcal F(s-1)=\varnothing\), and hence equality for the pair
\((s-1,n-s+1)\) gives
\(
\mathcal F(n-s+1)=\mathcal V(n-s+1).
\)
Choose \(X\in\mathcal V(s)\) and \(Y\in\mathcal V(n-s+1)\) with
\(\dim(X\cap Y)=1\). Then
\[
\Delta(X,Y)=s+(n-s+1)-2=n-1>d,
\]
because \(n\ge d+2\), contradicting \(\diam(\mathcal F)\le d\). Now Lemma~\ref{lem:equality-odd} applies through condition (b). Since
\(\mathcal F(0)\neq\varnothing\), its conclusion cannot hold for
\(\mathcal F^\perp\). Hence it must hold for \(\mathcal F\). But this would
force
\(
|\mathcal F(t+1)|=\qbinom{n-1}{t},
\)
contrary to the present assumption \(\mathcal F(t+1)=\varnothing\).
Thus equality cannot hold in this subcase.

This completes the proof in Case~2.

\paragraph{Case 3: There is no middle layer.}

We now assume
\begin{equation}\label{eq:no-middle-layer}
\supp(\mathcal F)\cap [r+1,n-t-1]=\varnothing.
\end{equation}

First suppose $d=2t$. 
Then $r=t$, so \eqref{eq:no-middle-layer} says that
\(
\mathcal F(k)=\varnothing\)
for \(t+1\le k\le n-t-1.\)
Hence
\[
|\mathcal F|
=
\sum_{k=0}^{t}
\bigl(|\mathcal F(k)|+|\mathcal F(n-k)|\bigr).
\]
By Corollary~\ref{lem:complementsum-1},
\(
|\mathcal F|
\le
\sum_{k=0}^{t}\qbinom{n}{k}.
\)
If equality holds, then
\(
|\mathcal F(k)|+|\mathcal F(n-k)|=\qbinom{n}{k}\)
for \(k=0,1,\ldots,t.\)
By Lemma~\ref{lem:equality}, either $\mathcal F$ or $\mathcal F^\perp$ satisfies one of the following:
if $n=d+1$, then for each $k=0,1,\ldots,t$, exactly one of the two complementary layers $\mathcal V(k)$ and $\mathcal V(n-k)$ is fully present; if $n\ge d+2$, then
\[
\mathcal F=\bigcup_{k=0}^{t}\mathcal V(k).
\]
These are exactly the equality cases in Theorem~\ref{thm:main}(1) and~(3).

Now suppose $d=2t+1$. 
Then $r=t+1$, so \eqref{eq:no-middle-layer} says that
\(
\mathcal F(k)=\varnothing\)
for \(t+2\le k\le n-t-1.\)
Thus
\[
|\mathcal F|
=
\sum_{k=0}^{t}
\bigl(|\mathcal F(k)|+|\mathcal F(n-k)|\bigr)
+
|\mathcal F(t+1)|.
\]
The family $\mathcal F(t+1)$ is $1$-intersecting in $V$ by Lemma~\ref{lem:intersecting}. 
By Theorem~\ref{thm:ekrvs},
\(
|\mathcal F(t+1)|\le \qbinom{n-1}{t}.
\)
Using Corollary~\ref{lem:complementsum-1}, we get
\[
|\mathcal F|
\le
\sum_{k=0}^{t}\qbinom{n}{k}
+
\qbinom{n-1}{t}.
\]
If equality holds, then
\(
|\mathcal F(k)|+|\mathcal F(n-k)|=\qbinom{n}{k}\)
for \(k=0,1,\ldots,t,\)
and
\(
|\mathcal F(t+1)|=\qbinom{n-1}{t},\) and
\(
\mathcal F(t+1)\) is \(1\)-intersecting in \(V.
\)
By Lemma~\ref{lem:equality-odd}, either $\mathcal F$ or $\mathcal F^\perp$ satisfies one of the following:
if $n=d+1$, then the complementary layers split as in Theorem~\ref{thm:main}(2), and the middle layer is a maximum $1$-intersecting family; if $n\ge d+2$, then
\[
\mathcal F
=
\bigcup_{k=0}^{t}\mathcal V(k)
\cup
\{A\in\mathcal V(t+1):X \le A\}
\]
for some fixed $1$-dimensional subspace $X \in \mathcal{V}(1)$. 
These are exactly the equality cases in Theorem~\ref{thm:main}(2) and~(4). This completes the proof.

\section{Proof of Theorem~\ref{thm:stab-typeA-even}}\label{sec:stab-even-A}
Let $m:=\min\supp(\mathcal F)$ and $M:=\max\supp(\mathcal F)$. 
Since the map $A\mapsto A^\perp$ is an isometry of $(\mathcal V,\Delta)$ and interchanges $\mathcal L_t$ and $\mathcal U_t$, $(\mathfrak E_A^{\mathrm{2t}},2t)$-admissibility is preserved under taking orthogonal complements. 
Thus, after possibly replacing $\mathcal F$ by $\mathcal F^\perp$, we may assume that
\begin{equation}\label{eq:typeA-even-normalize}
m+M\le n.
\end{equation}
Since $\mathcal F\not\subseteq \mathcal L_t=\bigcup_{i=0}^{t}\mathcal V(i)$, we have $M\ge t+1$.

We first note that
\(
M\le n-t-2.
\)
Indeed, if $m\ge t+2$, then \eqref{eq:typeA-even-normalize} gives $M\le n-m\le n-t-2$. 
If $m\le t+1$, then $M-m\le \diam(\mathcal F)\le 2t$, so
\(
M\le m+2t\le 3t+1\le n-t-2,
\)
where the last inequality follows from $n\ge 7t+5$.

We write
\[
g(n,t):=\sum_{i=0}^{t-1}\qbinom{n}{i}+\qbinom{n-1}{t-1}+\qbinom{n-1}{t}.
\]
We prove that $|\mathcal F|\le g(n,t)$, with equality only in the asserted cases.

First suppose that $M\ge t+2$. 
We shall show that then $|\mathcal F|<g(n,t)$. 
Since $|\mathcal F(k)|\le \qbinom{n}{k}$ for $0\le k\le t-1$, we have
\[
|\mathcal F|
\le
\sum_{i=0}^{t-1}\qbinom{n}{i}
+
|\mathcal F(t)|+|\mathcal F(t+1)|
+
\sum_{k=t+2}^{M}|\mathcal F(k)|.
\]

We first bound the tail. 
For each $k\ge t+2$, Lemma~\ref{lem:intersecting} implies that $\mathcal F(k)$ is $(k-t)$-intersecting in $V$. 
Since $M\le n-t-2$, we have $n\ge k+t+2$ for all $k\le M$, and hence Theorem~\ref{thm:ekrvs} gives
\[
|\mathcal F(k)|
\le
\max\left\{
\qbinom{n-k+t}{t},
\qbinom{k+t}{t}
\right\}.
\]
Splitting at $\lfloor n/2\rfloor$, we obtain
\begin{equation*}
\qbinom{n}{t}^{-1}
\sum_{t+2\le k\le \lfloor n/2\rfloor}|\mathcal F(k)|
\le
\sum_{k=t+2}^{\lfloor n/2\rfloor}
\frac{\qbinom{n-k+t}{t}}{\qbinom{n}{t}} <
\sum_{k=t+2}^{\lfloor n/2\rfloor}q^{-t(k-t)}
<
\frac{1}{q^t(q^t-1)}.
\end{equation*}
Similarly,
\begin{equation*}
\qbinom{n}{t}^{-1}
\sum_{\lfloor n/2\rfloor< k\le M}|\mathcal F(k)|
\le
\sum_{\lfloor n/2\rfloor< k\le M}
\frac{\qbinom{k+t}{t}}{\qbinom{n}{t}} <
\sum_{\lfloor n/2\rfloor< k\le M}q^{-t(n-k-t)}
<
\frac{1}{q^t(q^t-1)}.
\end{equation*}
Therefore
\begin{align*}
\sum_{k=t+2}^{M}|\mathcal F(k)|
<
\frac{2}{q^t(q^t-1)}\qbinom{n}{t}.
\end{align*}
Moreover,
\begin{equation*}
\frac{\frac{2}{q^t(q^t-1)}\qbinom{n}{t}}{\qbinom{n-1}{t}}
=
\frac{2}{q^t(q^t-1)}
\cdot
\frac{q^n-1}{q^{n-t}-1} <
\frac{2(q^t+1)}{q^t(q^t-1)}
\le \frac56,
\end{equation*}
where we used $q^t\ge 4$. Hence
\begin{equation}\label{eq:typeA-tail-final}
\sum_{k=t+2}^{M}|\mathcal F(k)|<\frac56\qbinom{n-1}{t}.
\end{equation}

Next we bound $\mathcal F(t)$ and $\mathcal F(t+1)$. Fix $A\in\mathcal F(M)$. If $\mathcal F(t)\neq\varnothing$, then for every $B\in\mathcal F(t)$, the diameter condition gives $\Delta(A,B)\le 2t$, and so
\[
2\dim(A\cap B)=M+t-\Delta(A,B)\ge M-t.
\]
Since $M\ge t+2$, every $B\in\mathcal F(t)$ meets $A$ nontrivially. 
Also, the existence of $B\in\mathcal F(t)$ gives $M-t\le 2t$, hence $M\le 3t$. 
Thus
\begin{equation}\label{eq:typeA-Ft-bound}
|\mathcal F(t)|
\le
\qbinom{M}{1}\qbinom{n-1}{t-1}
\le
\qbinom{3t}{1}\qbinom{n-1}{t-1}.
\end{equation}
The same bound is of course trivial if $\mathcal F(t)=\varnothing$.

Similarly, if $\mathcal F(t+1)\neq\varnothing$, then for every $B\in\mathcal F(t+1)$ we have
\[
2\dim(A\cap B)=M+t+1-\Delta(A,B)\ge M-t+1.
\]
Since $M\ge t+2$, this implies $\dim(A\cap B)\ge 2$. 
Moreover, the existence of a member of $\mathcal F(t+1)$ gives $M-(t+1)\le 2t$, hence $M\le 3t+1$. 
Therefore
\begin{equation}\label{eq:typeA-Ftplus-bound}
|\mathcal F(t+1)|
\le
\qbinom{M}{2}\qbinom{n-2}{t-1}
\le
\qbinom{3t+1}{2}\qbinom{n-2}{t-1}.
\end{equation}
Again this is trivial if $\mathcal F(t+1)=\varnothing$.

We now compare the two error terms with $\qbinom{n-1}{t}$. 
First,
\begin{equation}\label{eq:typeA-ratio1}
\frac{\qbinom{3t}{1}\qbinom{n-1}{t-1}}{\qbinom{n-1}{t}}
=
\qbinom{3t}{1}\frac{q^t-1}{q^{n-t}-1} \le
q^{3t}\frac{q^t}{q^{n-t-1}}
=
\frac{1}{q^{n-5t-1}}
\le \frac{1}{q^8}
\le \frac{1}{256}.
\end{equation}
Here we used $n\ge 7t+5$ and $t\ge 2$. 
Second,
\begin{equation}\label{eq:typeA-ratio2}
    \frac{\qbinom{3t+1}{2}\qbinom{n-2}{t-1}}{\qbinom{n-1}{t}}
=
\qbinom{3t+1}{2}
\frac{q^t-1}{q^{n-1}-1}
=
\frac{(q^{3t+1}-1)(q^{3t}-1)(q^t-1)}
{(q^2-1)(q-1)(q^{n-1}-1)} 
\le
\frac{q^{7t+1}}{(q^2-1)(q-1)q^{n-2}}
\le
\frac{1}{8}.
\end{equation}
Combining \eqref{eq:typeA-tail-final}, \eqref{eq:typeA-Ft-bound}, \eqref{eq:typeA-Ftplus-bound}, \eqref{eq:typeA-ratio1}, and \eqref{eq:typeA-ratio2}, we get
\[
|\mathcal F|
<
\sum_{i=0}^{t-1}\qbinom{n}{i}
+
\left(\frac56+\frac{1}{256}+\frac18\right)\qbinom{n-1}{t}
<
\sum_{i=0}^{t-1}\qbinom{n}{i}
+
\qbinom{n-1}{t}
<
g(n,t).
\]
Thus equality cannot occur when $M\ge t+2$.

We may therefore assume $M=t+1$. 
Then $\mathcal F(k)=\varnothing$ for all $k\ge t+2$, and $\mathcal F(t+1)\neq\varnothing$.

Suppose first that $\mathcal F(t+1)$ is not contained in any star. 
Fix $A\in\mathcal F(t+1)$. 
For every $B\in\mathcal F(t)$, we have
\[
2\dim(A\cap B)=2t+1-\Delta(A,B)\ge 1,
\]
so every member of $\mathcal F(t)$ meets $A$ nontrivially. Hence
\begin{equation}\label{eq:typeA-case2-Ft}
|\mathcal F(t)|
\le
\qbinom{t+1}{1}\qbinom{n-1}{t-1}.
\end{equation}
By Lemma~\ref{lem:intersecting}, the family $\mathcal F(t+1)$ is $1$-intersecting in $V$. 
Since it is not contained in any star,  Lemma~\ref{lem:nontrivial-1-intersecting} implies that
\[
|\mathcal F(t+1)|
\le
\qbinom{n-1}{t}
-
q^{t(t+1)}\qbinom{n-t-2}{t}
+
q^{t+1}
\le
\qbinom{t+1}{1}\qbinom{n-2}{t-1},
\]
where the last inequality follows from Lemma~\ref{lem:nontrivial-s-intersecting-2} when \(t\ge3\). When \(t=2\), it follows from the identity
\[
\qbinom{3}{1}\qbinom{n-2}{1}
-
\left(
\qbinom{n-1}{2}
-
q^6\qbinom{n-4}{2}
+
q^3
\right)
=
q(q+1)>0.
\]
Consequently,
\[
|\mathcal F|
\le
\sum_{i=0}^{t-1}\qbinom{n}{i}
+
2\qbinom{t+1}{1}\qbinom{n-1}{t-1}.
\]
Furthermore,
\begin{equation}\label{eq:typeA-case2-ratio}
\frac{\qbinom{t+1}{1}\qbinom{n-1}{t-1}}{\qbinom{n-1}{t}}
=
\qbinom{t+1}{1}\frac{q^t-1}{q^{n-t}-1} \le
q^{t+1}\frac{q^t}{q^{n-t-1}}
=
\frac{1}{q^{n-3t-2}}
\le
\frac{1}{q^{4t+3}}
\le
\frac{1}{2048}.
\end{equation}
It follows that
\[
2\qbinom{t+1}{1}\qbinom{n-1}{t-1}
\le
\frac{1}{1024}\qbinom{n-1}{t}
<
\qbinom{n-1}{t-1}+\qbinom{n-1}{t}.
\]
Thus $|\mathcal F|<g(n,t)$ in this case as well.

It remains to consider the case where $\mathcal F(t+1)$ is contained in a star. 
Choose $X\in\mathcal V(1)$ such that $X\le A$ for every $A\in\mathcal F(t+1)$. Assume first that there exists $B_0\in\mathcal F(t)$ with $X\not\le B_0$. 
For every $A\in\mathcal F(t+1)$, the diameter condition gives $\dim(A\cap B_0)\ge 1$. 
Since $X\le A$ but $X\not\le B_0$, the intersection $A\cap B_0$ contains a $1$-subspace of $B_0$. 
A union bound over the $1$-subspaces of $B_0$ gives
\begin{equation*}
|\mathcal F(t+1)|
\le
\qbinom{t}{1}\qbinom{n-2}{t-1}.
\end{equation*}
On the other hand, \eqref{eq:typeA-case2-Ft} still gives
\begin{equation*}
|\mathcal F(t)|
\le
\qbinom{t+1}{1}\qbinom{n-1}{t-1}.
\end{equation*}
Using the same comparison as in \eqref{eq:typeA-case2-ratio}, we get
\[
|\mathcal F(t)|+|\mathcal F(t+1)|
<
2\qbinom{t+1}{1}\qbinom{n-1}{t-1}
<
\qbinom{n-1}{t-1}+\qbinom{n-1}{t}.
\]
Therefore $|\mathcal F|<g(n,t)$.

We are left with the case where every member of $\mathcal F(t)$ contains $X$. 
Then
\[
|\mathcal F(t)|\le \qbinom{n-1}{t-1},\quad
|\mathcal F(t+1)|\le \qbinom{n-1}{t},
\]
and trivially $|\mathcal F(k)|\le \qbinom{n}{k}$ for $0\le k\le t-1$. 
Hence
\[
|\mathcal F|
\le
\sum_{i=0}^{t-1}\qbinom{n}{i}
+
\qbinom{n-1}{t-1}
+
\qbinom{n-1}{t}
=
g(n,t).
\]

If equality holds, then all these inequalities must be equalities. 
Thus $\mathcal F(k)=\mathcal V(k)$ for $0\le k\le t-1$, and
\[
\mathcal F(t)=\{A\in\mathcal V(t):X\le A\},\quad
\mathcal F(t+1)=\{A\in\mathcal V(t+1):X\le A\}.
\]
Therefore
\[
\mathcal F
=
\bigcup_{k=0}^{t-1}\mathcal V(k)
\cup
\{A\in\mathcal V(t):X\le A\}
\cup
\{A\in\mathcal V(t+1):X\le A\}.
\]
This is exactly $\mathcal B(X,t)$. Indeed, for $A\in\mathcal V(k)$ we have $\Delta(X,A)=1+k-2\dim(X\cap A)$, so $\Delta(X,A)\le t$ holds precisely when either $k\le t-1$, or $k\in\{t,t+1\}$ and $X\le A$.

Conversely, for every $X\in\mathcal V(1)$, the ball $\mathcal B(X,t)$ has diameter at most $2t$ by the triangle inequality. 
It is not contained in $\mathcal L_t$ because it contains $(t+1)$-subspaces, and it is not contained in $\mathcal U_t$ because it contains the zero subspace. 
Moreover,
\[
|\mathcal B(X,t)|
=
\sum_{i=0}^{t-1}\qbinom{n}{i}
+
\qbinom{n-1}{t-1}
+
\qbinom{n-1}{t}
=
g(n,t).
\]
Thus $\mathcal B(X,t)$ attains equality.

Finally, if we initially replaced $\mathcal F$ by $\mathcal F^\perp$ to ensure \eqref{eq:typeA-even-normalize}, the equality case translates back to $\mathcal F^\perp=\mathcal B(X,t)$ for some $X\in\mathcal V(1)$. 
Hence equality holds if and only if $\mathcal F=\mathcal B(X,t)$ or $\mathcal F^\perp=\mathcal B(X,t)$ for some $X\in\mathcal V(1)$.

\section{Proof of Theorem~\ref{thm:stab-type-odd}}\label{sec:stab-odd}
Set
\[
H(n,t):=\qbinom{n-1}{t}
-
q^{t(t+1)}\qbinom{n-t-2}{t}
+
q^{t+1}.
\]
Thus the claimed upper bound is
\(
\sum_{i=0}^{t}\qbinom{n}{i}+H(n,t).
\)
We first prove the result for $(\mathfrak E_A^{\mathrm{2t+1}},2t+1)$-admissible families. 
Let $\mathcal F\subseteq\mathcal V$ be $(\mathfrak E_A^{\mathrm{2t+1}},2t+1)$-admissible, and hence satisfy $\diam(\mathcal F)\le 2t+1$. 
Since taking orthogonal complements is an isometry and interchanges the two canonical sides, $(\mathfrak E_A^{\mathrm{2t+1}},2t+1)$-admissibility is preserved under replacing $\mathcal F$ by $\mathcal F^\perp$. 
As before, write $m:=\min\supp(\mathcal F)$ and $M:=\max\supp(\mathcal F)$. 
After possibly replacing $\mathcal F$ by $\mathcal F^\perp$, we may assume that $m+M\le n$.

Since $\mathcal F$ is not contained in any canonical double ball $\mathcal D_t(X)$, it is not contained in $\bigcup_{i=0}^{t}\mathcal V(i)$. 
Hence $M\ge t+1$.

We first show that the case $M\ge t+2$ is far below the desired bound. 
Under our normalization, we have $M\le n-t-2$. 
Indeed, if $m\ge t+2$, then $M\le n-m\le n-t-2$. 
If $m\le t+1$, then $M-m\le \diam(\mathcal F)\le 2t+1$, so
\(
M\le m+2t+1\le 3t+2\le n-t-2,
\)
where the last inequality follows from $n\ge 5t+3$. We decompose
\[
|\mathcal F|
\le
\sum_{i=0}^{t-1}\qbinom{n}{i}
+
|\mathcal F(t)|+|\mathcal F(t+1)|
+
\sum_{k=t+2}^{M}|\mathcal F(k)|.
\]
We first estimate the tail. 
For every $k\ge t+2$, Lemma~\ref{lem:intersecting} implies that $\mathcal F(k)$ is $(k-t)$-intersecting in $V$. 
Since $M\le n-t-2$, we have $n\ge k+t+2$ for every $k\le M$, so Theorem~\ref{thm:ekrvs} gives
\[
|\mathcal F(k)|
\le
\max\left\{
\qbinom{n-k+t}{t},
\qbinom{k+t}{t}
\right\}.
\]

Splitting at $\lfloor n/2\rfloor$, we obtain
\[
\qbinom{n}{t}^{-1}
\sum_{t+2\le k\le \lfloor n/2\rfloor}|\mathcal F(k)|
<
\sum_{k=t+2}^{\lfloor n/2\rfloor}q^{-t(k-t)}
<
\frac{1}{q^t(q^t-1)}.
\]
Similarly,
\[
\qbinom{n}{t}^{-1}
\sum_{\lfloor n/2\rfloor<k\le M}|\mathcal F(k)|
<
\sum_{\lfloor n/2\rfloor<k\le M}q^{-t(n-k-t)}
<
\frac{1}{q^t(q^t-1)}.
\]
Therefore
\[
\sum_{k=t+2}^{M}|\mathcal F(k)|
<
\frac{2}{q^t(q^t-1)}\qbinom{n}{t}.
\]
Moreover,
\[
\frac{\frac{2}{q^t(q^t-1)}\qbinom{n}{t}}{\qbinom{n-1}{t}}
=
\frac{2}{q^t(q^t-1)}
\cdot
\frac{q^n-1}{q^{n-t}-1}
<
\frac{2(q^t+1)}{q^t(q^t-1)}
\le \frac56,
\]
where we used $q^t\ge 4$. Hence
\begin{equation}\label{eq:odd-stab-tail}
\sum_{k=t+2}^{M}|\mathcal F(k)|
<
\frac56\qbinom{n-1}{t}.
\end{equation}

Next we bound $\mathcal F(t)$ and $\mathcal F(t+1)$.
If $\mathcal F(t)\neq\varnothing$, fix $A\in\mathcal F(M)$. 
For every $B\in\mathcal F(t)$, the diameter condition gives
\[
2\dim(A\cap B)=M+t-\Delta(A,B)\ge M-t-1.
\]
Since $M\ge t+2$, every member of $\mathcal F(t)$ meets $A$ nontrivially. 
Moreover, the existence of $B\in\mathcal F(t)$ gives $M-t\le 2t+1$, and hence $M\le 3t+1$. 
Thus
\begin{equation}\label{eq:odd-stab-Ft}
|\mathcal F(t)|
\le
\qbinom{M}{1}\qbinom{n-1}{t-1}
\le
\qbinom{3t+1}{1}\qbinom{n-1}{t-1}.
\end{equation}
The same bound is trivial if $\mathcal F(t)=\varnothing$. By Lemma~\ref{lem:intersecting}, the family $\mathcal F(t+1)$ is $1$-intersecting in $V$. 
Since $n\ge 5t+3\ge 2(t+1)+1$, Theorem~\ref{thm:EKR-vector-space-1} gives
\begin{equation}\label{eq:odd-stab-Ftplus}
|\mathcal F(t+1)|\le \qbinom{n-1}{t}.
\end{equation}
Finally,
\[
\frac{\qbinom{3t+1}{1}\qbinom{n-1}{t-1}}{\qbinom{n-1}{t}}
=
\qbinom{3t+1}{1}\frac{q^t-1}{q^{n-t}-1}
\le
q^{3t+1}\frac{q^t}{q^{n-t-1}}
=
\frac{1}{q^{n-5t-2}}
\le \frac12.
\]
Combining this with \eqref{eq:odd-stab-tail}, \eqref{eq:odd-stab-Ft}, and \eqref{eq:odd-stab-Ftplus}, we obtain
\[
|\mathcal F|
<
\sum_{i=0}^{t-1}\qbinom{n}{i}
+
\left(\frac56+\frac12+1\right)\qbinom{n-1}{t}
=
\sum_{i=0}^{t-1}\qbinom{n}{i}
+
\frac73\qbinom{n-1}{t}.
\]
Since
\[
\frac{\qbinom{n}{t}}{\qbinom{n-1}{t}}
=
\frac{q^n-1}{q^{n-t}-1}
>
q^t
\ge 4,
\]
we have $\frac73\qbinom{n-1}{t}<\qbinom{n}{t}$. Therefore
\(
|\mathcal F|<\sum_{i=0}^{t}\qbinom{n}{i}.
\)
This is strictly smaller than the asserted stability bound, since $H(n,t)>0$. 
Thus equality cannot occur when $M\ge t+2$.

It remains to consider the case $M=t+1$. 
Then $\mathcal F(k)=\varnothing$ for all $k\ge t+2$, and $\mathcal F(t+1)\neq\varnothing$. 
Since $\mathcal F$ is not contained in any $\mathcal D_t(X)$, the layer $\mathcal F(t+1)$ is not contained in any star. 
Indeed, if $\mathcal F(t+1)$ were contained in $\{A\in\mathcal V(t+1):X\le A\}$ for some $X\in\mathcal V(1)$, then every member of $\mathcal F$ would lie in
\[
\bigcup_{i=0}^{t}\mathcal V(i)
\cup
\{A\in\mathcal V(t+1):X\le A\}
=
\mathcal D_t(X),
\]
contrary to $(\mathfrak E_A^{\mathrm{2t+1}},2t+1)$-admissibility. By Lemma~\ref{lem:intersecting}, the family $\mathcal F(t+1)$ is $1$-intersecting in $V$. 
Since it is not contained in any star, Lemma~\ref{lem:nontrivial-1-intersecting}, applied with $k=t+1$, gives
\[
|\mathcal F(t+1)|
\le
\qbinom{n-1}{t}
-
q^{t(t+1)}\qbinom{n-t-2}{t}
+
q^{t+1}
=
H(n,t).
\]
Therefore
\[
|\mathcal F|
\le
\sum_{i=0}^{t}|\mathcal F(i)|+|\mathcal F(t+1)|
\le
\sum_{i=0}^{t}\qbinom{n}{i}+H(n,t).
\]
This proves the desired upper bound for $(\mathfrak E_A^{\mathrm{2t+1}},2t+1)$-admissible families.

If equality holds, then  all lower-layer inequalities must be equalities, and hence $\mathcal F(i)=\mathcal V(i)$ for every $0\le i\le t$. 
Also, $\mathcal F(t+1)$ must be an extremal nontrivial $1$-intersecting family in $\mathcal V(t+1)$. 
By Lemma~\ref{lem:nontrivial-1-intersecting}, one of the following occurs.

First, there exist $X\in\mathcal V(1)$ and $Y\in\mathcal V(t+1)$ with $X\not\le Y$ such that
\[
\mathcal F(t+1)=\mathcal{HM}(n,t+1,X,Y).
\]
In this case
\[
\mathcal F
=
\bigcup_{i=0}^{t}\mathcal V(i)\cup\mathcal{HM}(n,t+1,X,Y)
=
\mathcal K(n,t+1,X,Y).
\]
Second, if $t=2$, then there exists $Y\in\mathcal V(3)$ such that
\(
\mathcal F(3)=\mathcal{HM}^*(n,3,Y).
\)
In this case
\[
\mathcal F
=
\bigcup_{i=0}^{2}\mathcal V(i)\cup\mathcal{HM}^*(n,3,Y)
=
\mathcal K^*(n,3,Y).
\]
If the initial normalization was obtained by replacing $\mathcal F$ by $\mathcal F^\perp$, then the same conclusion gives $\mathcal F^\perp=\mathcal K(n,t+1,X,Y)$, or, when $t=2$, $\mathcal F^\perp=\mathcal K^*(n,3,Y)$.

Conversely, the listed families attain equality. 
Indeed, the families $\mathcal{HM}(n,t+1,X,Y)$ and $\mathcal{HM}^*(n,3,Y)$ are precisely the extremal nontrivial $1$-intersecting families given by Lemma~\ref{lem:nontrivial-1-intersecting}. 
Hence
\[
|\mathcal K(n,t+1,X,Y)|
=
\sum_{i=0}^{t}\qbinom{n}{i}+H(n,t),
\]
and the same formula holds for $\mathcal K^*(n,3,Y)$ when $t=2$.

It remains to check the diameter condition. 
The lower part $\bigcup_{i=0}^{t}\mathcal V(i)$ has diameter at most $2t$. 
Every member of the top layer has dimension $t+1$, and the top layer is $1$-intersecting, so two top-layer members have distance at most $2t$. 
Finally, if $A\in\mathcal V(i)$ with $i\le t$ and $B\in\mathcal V(t+1)$, then $\Delta(A,B)\le i+t+1\le 2t+1$. 
Thus all listed families have diameter at most $2t+1$. 
They are $(\mathfrak E_A^{\mathrm{2t+1}},2t+1)$-admissible because their top layers are not contained in any star, and they contain the zero subspace, so they are not contained in any orthogonal image of a canonical double ball. This completes the proof for Type~A.

Now let $\mathcal F$ be $(\mathfrak E_B^{\mathrm{2t+1}},2t+1)$-admissible. 
Since
\(
\mathcal D_t(X)=\mathcal B(\boldsymbol0,X,t)\)
with
\(\Delta(\boldsymbol0,X)=1,
\)
every canonical double ball is a union of two adjacent radius-\(t\) balls; the same is true for its orthogonal image. Hence \((\mathfrak E_B^{\mathrm{2t+1}},2t+1)\)-admissibility implies \((\mathfrak E_A^{\mathrm{2t+1}},2t+1)\)-admissibility.
Therefore the same upper bound and the same equality necessity hold for Type~B. 
It remains only to verify that the listed equality families are actually $(\mathfrak E_B^{\mathrm{2t+1}},2t+1)$-admissible. Let $\mathcal K$ be one of the listed families. 
Then $\bigcup_{i=0}^{t}\mathcal V(i)\subseteq\mathcal K$, and the top layer of $\mathcal K$ is not contained in any star. 
Suppose for contradiction that $\mathcal K\subseteq \mathcal B(P,Q,t)$ for some $P,Q\in\mathcal V$ with $\Delta(P,Q)=1$. 
Since $\boldsymbol 0\in\mathcal K$, at least one of $P,Q$ has dimension at most $t$. 
As $\Delta(P,Q)=1$, after relabelling we may write $P\le Q$, with $\dim P=a$ and $\dim Q=a+1$, where $a\le t$.

We claim that $a=0$. 
If $a\ge 1$, then $a+1\le t+1$. 
Since $n\ge 5t+3$, there exists a $t$-dimensional subspace $T$ with $T\cap Q=\boldsymbol 0$. 
Then
\[
\Delta(T,P)=t+a>t,\quad \Delta(T,Q)=t+a+1>t,
\]
contradicting $T\in\mathcal K\subseteq \mathcal B(P,Q,t)$. 
Thus $a=0$, so $P=\boldsymbol 0$ and $Q=Z$ for some $Z\in\mathcal V(1)$.

But then $\mathcal B(P,Q,t)=\mathcal B(\boldsymbol 0,Z,t)=\mathcal D_t(Z)$, whose $(t+1)$-layer is the star $\{A\in\mathcal V(t+1):Z\le A\}$. 
This contradicts the fact that the top layer of $\mathcal K$ is not contained in any star. 
Therefore $\mathcal K$ is not contained in any union of two adjacent radius-$t$ balls, and hence is $(\mathfrak E_B^{\mathrm{2t+1}},2t+1)$-admissible. 
The same conclusion holds for $\mathcal K^\perp$, because the orthogonal-complement map is an isometry and sends adjacent centers to adjacent centers. Thus the same extremal families attain equality for Type~B as well. 
The theorem follows.

\section{Proof of Theorem~\ref{thm:stab-typeB-even}}\label{sec:stab-even-B}
We shall use the following consequence of the nontrivial intersection theorems recorded in Lemmas~\ref{lem:nontrivial-s-intersecting} and~\ref{lem:nontrivial-s-intersecting-2}.
\begin{lemma}\label{lem:typeB-nontrivial-layer-bound}
Let $d=2t$ with $t\ge 2$, and assume $n\ge 6t$.
Let $\mathcal F\subseteq\mathcal V$ satisfy $\diam(\mathcal F)\le 2t$.
Assume that, after possibly replacing $\mathcal F$ by $\mathcal F^\perp$, its support satisfies $m+M\le n$, where $m:=\min\supp(\mathcal F)$ and $M:=\max\supp(\mathcal F)$. Let $0\le i\le 2t$ and put $k:=M-i$. If $k\ge t+1$ and $\mathcal F(k)$ is a nontrivial $(k-t)$-intersecting family in $V$, then
\(
|\mathcal F(k)|\le q^{4t+2}\qbinom{n}{t-1}.
\)
Moreover, if $M \ge t+1$ and $k=t$, then the same bound holds.
\end{lemma}
\begin{proof}[Proof of Lemma~\ref{lem:typeB-nontrivial-layer-bound}]
First suppose $k\ge t+1$ and $\mathcal F(k)$ is nontrivial $(k-t)$-intersecting. 
Set $s:=k-t$. 
Since $m+M\le n$ and $M-m\le 2t$, we have $M\le \lfloor n/2\rfloor+t$. 
Thus $k\le \lfloor n/2\rfloor+t$.

If $k\ge n/2$, then Theorem~\ref{thm:ekrvs} gives
\[
|\mathcal F(k)|\le \qbinom{k+t}{t}\le \qbinom{\lfloor n/2\rfloor+2t}{t}.
\]
Using the standard estimates
\(
\qbinom Nr\le q^{r(N-r+1)}\) and
\(
\qbinom Nr\ge q^{r(N-r)},
\)
we get
\[
\frac{\qbinom{\lfloor n/2\rfloor+2t}{t}}{\qbinom{n}{t-1}}
\le
q^{\,t(n/2+t+1)-(t-1)(n-t+1)}
\le q^{4t+2},
\]
where the last inequality follows from \(n\ge6t\) and \(t\ge2\). Hence $|\mathcal F(k)|\le q^{4t+2}\qbinom{n}{t-1}$ in this case.

Now suppose $k\le n/2-1$. 
Then $n\ge 2k+2$, so Lemmas~\ref{lem:nontrivial-s-intersecting} and~\ref{lem:nontrivial-s-intersecting-2} apply. 
Since $k-s=t$, we obtain
\[
|\mathcal F(k)|
\le
\max\left\{
\qbinom{t+1}{1}\qbinom{n-s-1}{t-1},
\qbinom{s+2}{1}\qbinom{n-s-1}{t-1}
\right\}.
\]
Furthermore, notice that
\(
\frac{\qbinom{n-s-1}{t-1}}{\qbinom{n}{t-1}}
\le q^{-(s+1)(t-1)}.
\)
Then for the first term,
\[
\qbinom{t+1}{1}q^{-(s+1)(t-1)}
\le q^{t+1}q^{-2(t-1)}
\le q^{3},
\]
while for the second term,
\[
\qbinom{s+2}{1}q^{-(s+1)(t-1)}
\le q^{s+2-(s+1)(t-1)}
\le q^{3},
\]
where the last inequality uses $t\ge 2$ and $s\ge 1$. 
Thus $|\mathcal F(k)|\le q^3\qbinom{n}{t-1}$, which is stronger than the claimed bound.

It remains to deal with the case $M\ge t+1$ and $k=t$. 
If $\mathcal F(t)=\varnothing$, there is nothing to prove. 
Otherwise, choose $A\in\mathcal F(M)$. 
Since $M-t\le 2t$, we have $M\le 3t$. 
For every $B\in\mathcal F(t)$,
\[
2\dim(A\cap B)=M+t-\Delta(A,B)\ge M-t.
\]
Since $M\ge t+1$, every $B\in\mathcal F(t)$ meets $A$ nontrivially. 
Consequently,
\[
|\mathcal F(t)|
\le
\qbinom{M}{1}\qbinom{n-1}{t-1}
\le
\qbinom{3t}{1}\qbinom{n-1}{t-1}
\le
q^{4t+2}\qbinom{n}{t-1}.
\]
This completes the proof.
    
\end{proof}

We now give the proof of Theorem~\ref{thm:stab-typeB-even}. 
Let $m:=\min\supp(\mathcal F)$ and $M:=\max\supp(\mathcal F)$. 
Since $(\mathfrak E_B^{\mathrm{2t}},2t)$-admissibility is preserved under taking orthogonal complements, after possibly replacing $\mathcal F$ by $\mathcal F^\perp$ we may assume $m+M\le n$. 
Also \(M-m\le 2t\), because \(\diam(\mathcal F)\le 2t\). Since \(\mathcal F\) is \((\mathfrak E_B^{\mathrm{2t}},2t)\)-admissible, it is not contained in \(\mathcal L_t=\mathcal B(\boldsymbol0,t)\). Hence \(M\ge t+1\).

We begin with two claims that provide the mechanism for turning a large trivial top layer into containment in a single ball.

\begin{claim}\label{claim:typeB-propagation}
Let $0\le i\le 2t$ and assume that $M-i\ge t+1$. 
Let $W\in\mathcal V(M-t-i)$ be such that $W\le A$ for every $A\in\mathcal F(M-i)$. 
If
\(
|\mathcal F(M-i)|>\qbinom{3t}{t}\qbinom{n}{t-1},
\)
then for every $0\le j\le 2t$ and every $B\in\mathcal F(M-j)$, we have
\[
\dim(B\cap W)\ge
\begin{cases}
\left\lceil M-t-\dfrac{i+j}{2}\right\rceil, & i+1\le j\le 2t,\\[2mm]
M-t-i, & 0\le j\le i.
\end{cases}
\]
\end{claim}
\begin{poc}
    Set $w:=M-t-i$ and $s_{ij}:=\left\lceil M-t-\frac{i+j}{2}\right\rceil$.
Then by Lemma~\ref{lem:intersecting}, the pair $(\mathcal F(M-i),\mathcal F(M-j))$ is cross-$s_{ij}$-intersecting.

First suppose $i+1\le j\le 2t$. 
Assume, for a contradiction, that there exists $B\in\mathcal F(M-j)$ with $\dim(B\cap W)=s_{ij}-\ell$ for some $\ell\ge 1$.
The cross-intersection condition also gives 
\(
\ell \le t.
\)
For every $A\in\mathcal F(M-i)$,
since $W\le A$ and $\dim(A\cap B)\ge s_{ij}$, we have
\begin{align*}
  \dim(A \cap (B+W))
  &\ge 
  \dim((A \cap B)+(A \cap W)) \\
  &\ge
  \dim(A\cap B)+\dim W-\dim(B\cap W)
  \ge 
  w+\ell.
\end{align*}
Thus every $A\in\mathcal F(M-i)$ contains $W$ and meets $B+W$ in dimension at least $w+\ell$. 
Therefore
\[
|\mathcal F(M-i)|
\le
\qbinom{\dim(B+W)-w}{\ell}
\qbinom{n-w-\ell}{t-\ell}.
\]
Since
\(
\dim(B+W)-w=t+\ell-j+\left\lfloor\frac{i+j}{2}\right\rfloor\le 2t
\)
and $1\le \ell\le t$, we get
\[
|\mathcal F(M-i)|
\le
\qbinom{2t}{\ell}\qbinom{n}{t-\ell}
\le
\qbinom{2t}{t}\qbinom{n}{t-1}
<
\qbinom{3t}{t}\qbinom{n}{t-1},
\]
a contradiction. 

Now suppose $0\le j\le i$. 
Assume that there exists $B\in\mathcal F(M-j)$ with $\dim(B\cap W)=w-\ell$ for some $\ell\ge 1$. 
The cross-intersection condition also gives
\(
\ell\le t-i+\left\lfloor\frac{i+j}{2}\right\rfloor.
\)
For every $A\in\mathcal F(M-i)$, we again have
\begin{align*}
  \dim(A \cap (B+W))
  &\ge 
  \dim((A \cap B)+(A \cap W)) \\
  &\ge
  \dim(A\cap B)+\dim W-\dim(B\cap W)
  \ge 
  s_{ij}+\ell.
\end{align*}
Hence
\[
|\mathcal F(M-i)|
\le
\qbinom{\dim(B+W)-w}{s_{ij}+\ell-w}
\qbinom{n-s_{ij}-\ell}{M-i-s_{ij}-\ell}.
\]
Here
\(
\dim(B+W)-w=t+i-j+\ell,
\)
and
\(
s_{ij}+\ell-w=\ell+i-\left\lfloor\frac{i+j}{2}\right\rfloor.
\)
Consequently the first Gaussian coefficient is bounded by $\qbinom{3t}{t}$, while the second is bounded by $\qbinom{n}{t-1}$. 
Thus
\[
|\mathcal F(M-i)|
\le
\qbinom{3t}{t}\qbinom{n}{t-1},
\]
again a contradiction. This finishes the proof.
\end{poc}
\begin{claim}\label{claim:typeB-trivial-layer-small}
Let $1\le i\le 2t$ and assume that $M-i\ge t+1$. 
Let $W\in\mathcal V(M-t-i)$ be such that $W\le A$ for every $A\in\mathcal F(M-i)$. 
Then
\[
|\mathcal F(M-i)|\le \qbinom{3t}{t}\qbinom{n}{t-1}.
\]
\end{claim}
\begin{poc}
Suppose not. 
Choose $B\in\mathcal F(M)$. 
Applying Claim~\ref{claim:typeB-propagation} with $j=0$, we obtain $W\le B$. 
By Lemma~\ref{lem:intersecting}, the pair $(\mathcal F(M-i),\mathcal F(M))$ is cross-$\left\lceil M-t-\frac{i}{2}\right\rceil$-intersecting. Hence every $A\in\mathcal F(M-i)$ contains $W$ and meets $B$ in dimension at least $\left\lceil M-t-\frac{i}{2}\right\rceil$.
Counting such subspaces gives
\[
|\mathcal F(M-i)|
\le
\qbinom{t+i}{\left\lceil i/2\right\rceil}
\qbinom{n-\left\lceil M-t-i/2\right\rceil}{t-\left\lceil i/2\right\rceil}
\le
\qbinom{3t}{t}\qbinom{n}{t-1},
\]
a contradiction.
\end{poc}

By Lemma~\ref{lem:intersecting}, the top layer $\mathcal F(M)$ is $(M-t)$-intersecting in $V$. We then focus on this layer.

We first consider the case where $\mathcal F(M)$ is a trivial $(M-t)$-intersecting family and
\(
|\mathcal F(M)|>\qbinom{3t}{t}\qbinom{n}{t-1}.
\)
Then there exists \(W\in\mathcal V(M-t)\) such that \(W\le A\) for every \(A\in\mathcal F(M)\). Applying Claim~\ref{claim:typeB-propagation} with \(i=0\), we get, for every $0\le j\le 2t$ and every $B\in\mathcal F(M-j)$,
\(
\dim(B\cap W)\ge \left\lceil M-t-\frac{j}{2}\right\rceil.
\)
Therefore
\[
\Delta(B,W)
=
(M-j)+(M-t)-2\dim(B\cap W)
\le t.
\]
Since $M-m\le 2t$, every nonempty layer of $\mathcal F$ is of the form $\mathcal F(M-j)$ for some $0\le j\le 2t$. 
Thus $\mathcal F\subseteq \mathcal B(W,t)$, contradicting $(\mathfrak E_B^{\mathrm{2t}},2t)$-admissibility.

Hence the above case cannot occur. 
We now bound every layer of dimension at least $t$. For the top layer, either $\mathcal F(M)$ is trivial and
\(
|\mathcal F(M)|\le \qbinom{3t}{t}\qbinom{n}{t-1},
\)
or $\mathcal F(M)$ is nontrivial, in which case Lemma~\ref{lem:typeB-nontrivial-layer-bound} gives
\(
|\mathcal F(M)|\le q^{4t+2}\qbinom{n}{t-1}.
\)
Thus
\[
|\mathcal F(M)|
\le
\max\left\{
\qbinom{3t}{t}\qbinom{n}{t-1},
q^{4t+2}\qbinom{n}{t-1}
\right\}.
\]
Now let $1\le i\le 2t$ with $M-i\ge t+1$ and $\mathcal F(M-i)\neq\varnothing$. 
If $\mathcal F(M-i)$ is a trivial $(M-t-i)$-intersecting family, then Claim~\ref{claim:typeB-trivial-layer-small} gives
\(
|\mathcal F(M-i)|\le \qbinom{3t}{t}\qbinom{n}{t-1}.
\)
If it is nontrivial, then Lemma~\ref{lem:typeB-nontrivial-layer-bound} gives
\(
|\mathcal F(M-i)|\le q^{4t+2}\qbinom{n}{t-1}.
\)
Finally, if $\mathcal F(t)\neq\varnothing$, Lemma~\ref{lem:typeB-nontrivial-layer-bound} gives the same bound
\(
|\mathcal F(t)|\le q^{4t+2}\qbinom{n}{t-1}.
\)
Since $\qbinom{3t}{t}>q^{2t^2}$ and $t\ge 2$, we have
\(
q^{4t+2}<q^2\qbinom{3t}{t}.
\)
Consequently every nonempty layer $\mathcal F(k)$ with $k\ge t$ satisfies
\[
|\mathcal F(k)|<q^2\qbinom{3t}{t}\qbinom{n}{t-1}.
\]
There are at most \(2t+1\) nonempty layers of dimension at least \(t\), because all nonempty dimensions lie in the interval \([m,M]\) and \(M-m\le2t\).
Therefore
\[
\sum_{k=t}^{M}|\mathcal F(k)|
<
(2t+1)q^2\qbinom{3t}{t}\qbinom{n}{t-1}.
\]
For the lower layers we use the trivial estimate $|\mathcal F(k)|\le \qbinom{n}{k}$ for $0\le k\le t-1$. 
Hence
\[
|\mathcal F|
=
\sum_{k=m}^{t-1}|\mathcal F(k)|+\sum_{k=t}^{M}|\mathcal F(k)|
<
\sum_{i=0}^{t-1}\qbinom{n}{i}
+
(2t+1)q^2\qbinom{3t}{t}\qbinom{n}{t-1}.
\]
This finishes the proof.

\section*{Acknowledgements}

This research originated during the 35th KIAS Combinatorics Workshop, held in Busan, Korea, in December 2025. The authors would like to thank the organizers of the workshop for creating a stimulating and productive environment that made this collaboration possible. Chenhui Lv warmly thanks Jongyook Park for his generous invitation and kind hospitality during the visit. During the early stages of this work, Zixiang Xu was a senior researcher at Extremal Combinatorics and Probability Group, Institute for Basic Science, and was supported by IBS-R029-C4. Zixiang Xu would like to thank Jiaqi Liao for kindly sending him an updated version of~\cite{LLY2026}.

\bibliographystyle{abbrv}
\bibliography{Kleitman}
\end{document}